\documentclass[a4, 12pt]{amsart}
\usepackage[mathscr]{eucal}
\usepackage{amssymb}
\usepackage{latexsym}
\usepackage{amsthm}
\usepackage{color,ulem}
\theoremstyle{plain}
\newtheorem{theorem}{Theorem}[section]

\newtheorem{remark}{Remark}[section]
\newtheorem{lemma}{Lemma}[section]

\makeatletter
\@addtoreset{equation}{section}

\title[Complete $\lambda$-translators in $\mathbb{R}^{4}_{1}$]
{Complete $3$-dimensional $\lambda$-translators in the Minkowski space $\mathbb{R}^{4}_{1}$}
\author [Z. Li and G. Wei]{Zhi Li and Guoxin Wei}
\address{Zhi Li \\  School of Mathematical Sciences, South China Normal University,
\newline \indent 510631, Guangzhou,  China.  \newline \indent lizhihnsd@126.com}
\address{Guoxin Wei \\  School of Mathematical Sciences, South China Normal University,
\newline \indent 510631, Guangzhou,  China. \newline \indent  weiguoxin@tsinghua.org.cn}

\begin{document}
\maketitle

\begin{abstract}
In this paper, we obtain the classification theorem for three-dimensional complete space-like $\lambda$-translators $x:M^{3} \rightarrow \mathbb R^{4}_{1}$  with constant norm of the second fundamental form and constant $f_{4}$ in the Minkowski space $\mathbb R^{4}_{1}$.
\end{abstract}

\footnotetext{2010 \textit{Mathematics Subject Classification}:
53C44, 53C40.}
\footnotetext{{\it Key words and phrases}:
the second fundamental form, $\lambda$-translator, the generalized maximum principle.}

\footnotetext{
The second author was partly supported by grant No. 11771154 of NSFC, by grant No. 2019A1515011451 of Natural Science Foundation of Guangdong Province
and by GDUPS (2018).}

\section{introduction}
\vskip2mm
\noindent
\noindent
Let $x:M^{n} \rightarrow \mathbb{R}^{n+1}_{1}$ be an immersed space-like hypersurface in the Minkowski space $\mathbb{R}^{n+1}_{1}$.
Fix a constant vector $T \neq 0 $ in $\mathbb{R}^{n+1}_{1}$ and $\lambda$ a real number. In this paper we study
orientable hypersurface $M^{n}$ in $\mathbb{R}^{n+1}_{1}$ whose mean curvature vector $\vec{H}$ satisfies
\begin{equation}\label{1.1-1}
\vec{H}+ T^{\perp}=\lambda \mathbf{n},
\end{equation}
where $\vec{H}= H\mathbf{n}$ and $\mathbf{n}$ is the unit normal
vector. Then $x$ is called a $\lambda$-translating soliton or simply a $\lambda$-translator of the mean curvature flow (MCF).
The constant vector $T$ will be called the corresponding translating vector or density
vector. In particular, if we denote by $\langle\cdot, \cdot\rangle$ the standard inner product on $\mathbb{R}^{n+1}_{1}$,
then the equation \eqref{1.1-1} is equivalent  to
\begin{equation}\label{1.1-2}
H-\langle T, \mathbf{n}\rangle=\lambda, \ \ \langle \mathbf{n}, \mathbf{n}\rangle=-1,
\end{equation}
where $x$ is space-like.

The interest of this equations is due to its relation with
manifolds with density. So it naturally makes sense to study the $\lambda$-translator in $\mathbb{R}^{n+1}_{1}$.
Indeed, as described in \cite{L2}, considering $\mathbb{R}^{n+1}_{1}$ with a positive
density function $e^{\phi}, \phi \in C^{\infty}(\mathbb{R}^{n+1}_{1})$, which serves as a weight for the volume
and the surface area. The first variation of the weighted volume $V_{\phi}(t)$ with density $e^{\phi}$
under compactly supported variations of $M^{n}$ is
\begin{equation}\label{1.1-3}
\left.\frac{d}{dt}\right|_{t=0}V_{\phi}(t)=-\int_{M^{n}}H_{\phi}\langle \mathbf{n}, \mathbf{\xi}\rangle dV_{\phi},
\end{equation}
where $\mathbf{\xi}$ is the variation vector and $H_{\phi}=H-\langle \nabla\phi, \mathbf{n} \rangle$.
So that $X$ is a critical point of the functional $V_{\phi}(t)$ for a given weighted volume if and only if $H_{\phi}$ is a constant function $H_{\phi}\equiv\lambda$: see (\cite{GM}, \cite{M}).
In particular, if we take $\phi : \mathbb{R}^{n+1}_{1}\rightarrow \mathbb{R}$ to be the height function $\phi(x) := \langle x, T\rangle$ , then the expression $H_{\phi}\equiv\lambda$ is
exactly the equation \eqref{1.1-2}. Secondly, a special case of \eqref{1.1-1} is when
$\lambda=0$. In such a case the immersion $x$ is called a translating soliton
of the mean curvature flow, or simply a translator (\cite{W1}).
Translators play an important role in the study of mean curvature flow.
On the one hand, a translating soliton is a solution of the mean curvature flow that
evolves purely by translations along the direction $T$.
On the other hand, they arise as blow-up solutions of MCF at type II singularities(\cite{HS1}, \cite{I}).
For instance, Huisken and Sinestrari (\cite{HS2})
proved that at type II singularity of a mean convex flow,
there exists a blow-up solution which is a convex translating solution.
Besides, in the
nonparametric form, the equation $H_{\phi}=0$ appeared in the classical article of Serrin
(\cite{S1}) and it was studied in the context of the maximum principle of elliptic
equations.
As we know, translating soliton have
been widely studied and various interesting results have been obtained in recent years.
For more information about translating soliton, please refer to the literatures (\cite{AW}, \cite{CSS}, \cite{H}, \cite{HR}, \cite{IR}, \cite{MH}, \cite{MSS}, \cite{P}, \cite{S}, \cite{S2}, \cite{W2}, \cite{X}).

In \cite{L1},  L\'{o}pez classified all $\lambda$-translators in $\mathbb{R}^{3}$
that are invariant by a one-parameter group of translations and a one-parameter group of rotations.
He also studied in \cite{L2} the shape of a compact $\lambda$-translator of $\mathbb{R}^{3}$
in terms of the geometry of its boundary, obtaining some necessary conditions for the existence of two-dimensional
compact $\lambda$-translators with a given closed boundary curve. In particular, he proved that
there do not exist any closed $\lambda$-translators of dimension two.
In fact, just as that $\lambda$-translating solitons are generalization of the translators of mean curvature flow, $\lambda$-hypersurfaces defined by Cheng and
 Wei in \cite{CW1} are generalization of the self-shrinkers of mean curvature flow. Self-shrinking solutions are important in the study of type-I singularities of
MCF. For instance, by proving the monotonicity formula, at a given type-I singularity of the
MCF, Huisken \cite{H1} proved that the flow is asymptotically self-similar, which implies that in this
situation the flow can be modeled by self-shrinking solutions. As is known, there have been many rigidity theorems and classification theorems for self-shrinkers
in the Euclidean space and the pseudo-Euclidean space.
Furthermore, there have been, up to now, several interesting and important results in the
study of $\lambda$-hypersurfaces. In particular, Cheng and Wei recently obtained a classification
theorem using their own generalized maximum principle (\cite{CW}) specially for $\lambda$-hypersurfaces,
which generalizes an interesting classification theorem in \cite{CO} for self-shrinkers.

The classification theorem also exists for $\lambda$-translators. For example, canonical examples of $\lambda$-translators in $\mathbb{R}^{n+1}_{1}$ are the space-like affine hyperplanes, and the right hyperbolic cylinders $\mathbb{H}^{k}(r) \times \mathbb{R}^{n-k}$ with $1 \leq k \leq n-1$, where
$\mathbb{H}^{k}(r)$ is the hyperbolic $k$-space defined by
$$\mathbb{H}^{k}(r)=\{x\in \mathbb{R}^{k+1}_{1}; \langle x, x\rangle = -{\color{red}r}^{2}\}.$$
Recently, Li, Qiao and Liu \cite{LQL} have classified  complete $\lambda$-translators in the Euclidean space $\mathbb{R}^{3}$ and the Minkowski space $\mathbb{R}^{3}_{1}$ with second fundamental form of constant length $S$.
For the higher dimension $n$, it is not easy to classify $\lambda$-translator in $\mathbb R^{n}$ and $\mathbb R^{n}_{1}$ with constant squared norm $S$ of the second fundamental form. In this paper, under the assumption that $f_{4}$ is constant, we give a complete classification for $3$-dimensional complete $\lambda$-translators in $\mathbb R^{4}_{1}$ with constant squared norm $S$ of the second fundamental form.
In fact, we prove the following result.
\begin{theorem}\label{theorem 1.1}
 Let $x: M^{3}\to \mathbb{R}^{4}_{1}$ be a
$3$-dimensional complete space-like $\lambda$-translator in $\mathbb R^{4}_{1}$.
If the squared norm $S$ of the second fundamental form and $f_{4}$ are constant, then $x: M^{3}\to \mathbb{R}^{4}_{1}$ is
isometric to one of
\begin{enumerate}
\item $\mathbb {R}^{3}_{1}$,
\item $\mathbb{H}^{1}(\frac{1}{\lambda})\times \mathbb{R}^{2}$,
\item $\mathbb{H}^{2}(\frac{2}{\lambda})\times \mathbb{R}^{1}$.
\end{enumerate}
In particular, $S$ must be $0$, $\lambda^{2}$ and $\frac{1}{2}\lambda^{2}$; $f_4$ must be $0$, $\lambda^{4}$ and $\frac{1}{8}\lambda^{4}$, where $\lambda \neq 0$, $S=\sum\limits_{i,j}h_{ij}^2$ and  $f_{4}=\sum\limits_{i,j,k,l}h_{ij}h_{jk}h_{kl}h_{li}$.
\end{theorem}
\begin{remark}
We also obtain a similar classification for $3$-dimensional complete $\lambda$-translators in $\mathbb R^{4}$(see \cite{LCW20}). That is, for a $3$-dimensional complete $\lambda$-translator in the Euclidean space $\mathbb{R}^{4}$, if the squared norm $S$ of the second fundamental form and $f_{4}$ are constant, then
hypersurface is isometric to one of
$\mathbb {R}^{3}$;
$\mathbb{S}^{1}(\frac{1}{\lambda})\times \mathbb{R}^{2}$;
$\mathbb{S}^{2}(\frac{2}{\lambda})\times \mathbb{R}^{1}$.
\end{remark}

\vskip5mm
\section {Preliminaries}
\vskip2mm

\noindent
Let $x: M^{n} \rightarrow\mathbb{R}^{n+1}_{1}$ be an
$n$-dimensional  space-like hypersurface of the $(n+1)$-dimensional Minkowski space $\mathbb{R}^{n+1}_{1}$. Around each point of
$M^{n}$, we choose a local orthonormal frame field
$\{e_{A}\}_{A=1}^{n+1}$ in $\mathbb{R}^{n+1}_{1}$ with dual coframe field
$\{\omega_{A}\}_{A=1}^{n+1}$, such that, restricted to $M^{n}$,
$e_{1},\cdots, e_{n}$ are tangent on $M^{n}$.

\noindent
From now on,  we use the following conventions on the ranges of indices:
$$
 1\leq i,j,k,l\leq n
$$
and $\sum_{i}$ means taking  summation from $1$ to $n$ for $i$.
Then we have
\begin{equation*}
dx=\sum_{i}\limits \omega_{i} e_{i},
\end{equation*}
\begin{equation*}
de_{i}=\sum_{j}\limits \omega_{ij} e_{j}+\omega_{in+1} e_{n+1},
\end{equation*}
\begin{equation*}
de_{n+1}=\omega_{n+1i}e_{i}, \ \ \omega_{n+1i}=\omega_{in+1},
\end{equation*}
where $\omega_{ij}=-\omega_{ji}$ is the Levi-Civita connection of the hypersurface.

\noindent By  restricting  these forms to $M^{n}$,  we get
\begin{equation}\label{2.1-1}
\omega^{n+1}=0.
\end{equation}

\noindent Taking exterior derivatives of \eqref{2.1-1}, we obtain
\begin{equation*}
0=d\omega_{n+1}=\sum_{i}\omega_{n+1i}\wedge\omega_{i}.
\end{equation*}
By Cartan's lemma, we know that there exist local smooth functions $h_{ij}$ ,
$1\leq  i,j \leq n$, such that
\begin{equation}\label{2.1-2}
\omega_{in+1}=\sum_{j} h_{ij}\omega_{j},\quad
h_{ij}=h_{ji}.
\end{equation}

$$
h=\sum_{i,j}h_{ij}\omega_i\otimes\omega_{j},\quad  H= \sum_i\limits h_{ii}
$$
are called  the second fundamental form and the mean curvature of $x: M\rightarrow\mathbb{R}^{n+1}_{1}$, respectively.
Let $S=\sum_{i,j}\limits (h_{ij})^2$ be  the squared norm
of the second fundamental form  of $x: M\rightarrow\mathbb{R}^{n+1}_{1}$.
The induced structure equations of $M^{n}$ are given by
\begin{equation*}
d\omega_{i}=\sum_j \omega_{ij}\wedge\omega_{j}, \quad \omega_{i j}=-\omega_{ji},
\end{equation*}
\begin{equation*}
d\omega_{ij}=\sum_{k} \omega_{ik}\wedge\omega_{k j}-\frac{1}{2}\sum_{k,l}
R_{ijkl} \omega_{k}\wedge\omega_{l},
\end{equation*}
where $R_{ijkl}$ denotes components of the curvature tensor of the hypersurface.
Hence, the Gauss equations of the space-like hypersurface $x$ in $\mathbb{R}^{n+1}_{1}$ are as follows:
\begin{equation}\label{2.1-3}
R_{ijkl}=-(h_{ik}h_{jl}-h_{il}h_{jk}).
\end{equation}

\noindent
Defining the
covariant derivative of $h_{ij}$ by
\begin{equation}\label{2.1-4}
\sum_{k}h_{ijk}\omega_{k}=dh_{ij}+\sum_{k}h_{kj}\omega_{ki}+\sum_{k}h_{ik}\omega_{kj},
\end{equation}
we obtain the Codazzi equations
\begin{equation}\label{2.1-5}
h_{ijk}=h_{ikj}.
\end{equation}
By taking exterior differentiation of \eqref{2.1-4}, and
defining
\begin{equation}\label{2.1-6}
\sum_{l}h_{ijkl}\omega_{l}=dh_{ijk}+\sum_{l}h_{ljk}\omega_{li}
+\sum_{l}h_{ilk}\omega_{lj}+\sum_{l}h_{ijl}\omega_{lk},
\end{equation}
we have the following Ricci identities:
\begin{equation}\label{2.1-7}
h_{ijkl}-h_{ijlk}=\sum_{m}h_{mj}R_{mikl}+\sum_{m} h_{im}R_{mjkl}.
\end{equation}
Defining
\begin{equation}\label{2.1-8}
\begin{aligned}
\sum_{m}h_{ijklm}\omega_{m}&=dh_{ijkl}+\sum_{m}h_{mjkl}\omega_{mi}
+\sum_{m}h_{imkl}\omega_{mj}+\sum_{m}h_{ijml}\omega_{mk}\\
&+\sum_{m}h_{ijkm}\omega_{ml}
\end{aligned}
\end{equation}
and taking exterior differentiation of \eqref{2.1-6}, we get
\begin{equation}\label{2.1-9}
\begin{aligned}
h_{ijkln}-h_{ijknl}&=\sum_{m} h_{mjk}R_{miln}
+ \sum_{m}h_{imk}R_{mjln}+ \sum_{m}h_{ijm}R_{mkln}.
\end{aligned}
\end{equation}
For a smooth function $f$, we define
\begin{equation}\label{2.1-10}
\sum_i f_{,i}\omega_i=df,
\end{equation}
\begin{equation}\label{2.1-11}
\sum_j f_{,ij}\omega_j=df_{,i}+\sum_j
f_{,j}\omega_{ji},
\end{equation}
\begin{equation}\label{2.1-12}
|\nabla f|^2=\sum_{i }(f_{,i})^2,\ \ \ \  \Delta f =\sum_i f_{,ii}.
\end{equation}

\noindent Let $V$ be a tangent $C^{1}$-vector field on $M^{n}$, and denote by $Ric_{V} := Ric-\frac{1}{2}L_{V}g$ the
Bakry-Emery Ricci tensor with $L_{V}$ to be the Lie derivative along the vector field $V$. Define a
differential operator
\begin{equation*}
\Delta_{V}f=\Delta f+\langle V,\nabla f\rangle,
\end{equation*}
where $\Delta$ and $\nabla$ denote the Laplacian and the gradient
operator, respectively. Then we have the following maximum principle
of Omori-Yau type which was proved by Chen-Qiu \cite{CQ} and  Li-Qiao-Liu\cite{LQL}:

\vskip2mm
\noindent

\begin{lemma}\label{lemma 2.1}
Let $(M^{n}, g)$ be a complete Riemannian manifold, and $V$ is a $C^{1}$ vector field on $M^{n}$. If the Bakry-Emery Ricci tensor $Ric_{V}$ is bounded from below, then for any $f\in C^{2}(M^{n})$ bounded from above, there exists a sequence ${p_{m}} \subset M^{n}$, such that
\begin{equation*}
\lim_{m\rightarrow\infty} f(p_{m})=\sup f,\quad
\lim_{m\rightarrow\infty} |\nabla f|(p_{m})=0,\quad
\lim_{m\rightarrow\infty}\Delta_{V}f(p_{m})\leq 0.
\end{equation*}
\end{lemma}

\noindent Suppose that the given hypersurface $x: M\rightarrow\mathbb{R}^{n+1}_{1}$
is a $\lambda$-translator with a translating
vector $T$, and let $\{e_{i}\}$ be an orthonormal tangent frame on $M^{n}$. Then from the definitions
\eqref{1.1-2} and \eqref{1.1-3} of $\lambda$-translators in $\mathbb{R}^{n+1}_{1}$, we have the following basic formulas for covariant
derivatives:

\begin{equation}\label{2.1-13}
\aligned
\nabla_{i}H
=&\sum_{k}h_{ik}\langle T, e_{k}\rangle, \\
\nabla_{j}\nabla_{i}H
=&\sum_{k}h_{ijk}\langle  T, e_{k}\rangle+(H-\lambda)\sum_{k}h_{ik}h_{kj},\\
\nabla_{l}\nabla_{j}\nabla_{i}H
=&\sum_{k}h_{ijkl}\langle  T, e_{k}\rangle+(H-\lambda)\sum_{k}(h_{ikl}h_{kj}+h_{ik}h_{kjl}+h_{ijk}h_{kl})\\
&+\nabla_{l}H\sum_{k}h_{ik}h_{kj}.
\endaligned
\end{equation}

\noindent
Moreover, we define three functions $f_3$, $f_4$  and $f_5$ as follows:
$$
f_{3}=\sum_{i,j,k}h_{ij}h_{jk}h_{ki}, \ \
f_{4}=\sum_{i,j,k,l}h_{ij}h_{jk}h_{kl}h_{li},\ \
f_{5}=\sum\limits_{i,j,k,l,m}h_{ij}h_{jk}h_{kl}h_{lm}h_{mi}.
$$
If we denote $V = T^{T}$, the tangent component of the translating vector $T$ when
restricted to $M^{n}$,
then direct computations using above formulas and the Ricci identities easily give the
following Lemma (cf. \cite{CW1} and \cite{LQL}):

\begin{lemma}\label{lemma 2.2}
Let $x: M^{n}\rightarrow\mathbb{R}^{n+1}_{1}$ be an $n$-dimensional complete $\lambda$-translator in $\mathbb R^{n+1}_{1}$, we have
\begin{equation}\label{2.1-14}
\Delta_{-V}H=S(H-\lambda).
\end{equation}
\begin{equation}\label{2.1-15}
\frac{1}{2}\Delta_{-V}H^{2}=|\nabla H|^{2}+S(H-\lambda)H.
\end{equation}
\begin{equation}\label{2.1-16}
\frac{1}{2}\Delta_{-V}S=\sum_{i,j,k}h_{ijk}^{2}+S^{2}-\lambda f_{3}.
\end{equation}
\begin{equation}\label{2.1-17}
\aligned
\frac{1}{4}\Delta_{-V}f_{4}
=&2\sum_{i,j,k,l,m}h_{ijm}h_{jkm}h_{kl}h_{li}+\sum_{i,j,k,l,m}h_{ijm}h_{jk}h_{klm}h_{li}\\
 &+Sf_{4}-\lambda f_{5}.
\endaligned
\end{equation}
\end{lemma}

\noindent
\begin{lemma}\label{lemma 2.3}
Let $x: M^{n}\rightarrow\mathbb{R}^{n+1}_{1}$ be an $n$-dimensional complete $\lambda$-translator in $\mathbb R^{n+1}_{1}$. If $S$ is constant, we have
\begin{equation}\label{2.1-18}
\aligned
\frac{1}{2}\Delta_{-V}\sum_{i, j,k}(h_{ijk})^{2}
=&\sum_{i,j,k,l}(h_{ijkl})^{2}+S\sum_{i,j,k}(h_{ijk})^{2}-6\sum_{i,j,k,l,p}h_{ijk}h_{il}h_{jp}h_{klp}\\
&+3\sum_{i,j,k,l,p}h_{ijk}h_{ijl}h_{kp}h_{lp}-3\lambda\sum_{i,j,k,l}h_{ijk}h_{ijl}h_{kl}.
\endaligned
\end{equation}
Furthermore, for $n=3$, we have
$$f_{3}=\frac{H}{2}(3S-H^{2})-3h_{11}h^{2}_{23}-3h_{22}h^{2}_{13}-3h_{33}h^{2}_{12}+3h_{11}h_{22}h_{33}+6h_{12}h_{13}h_{23}.$$
Then,
\begin{equation}\label{2.1-19}
\aligned
 &\frac{1}{2}\Delta_{-V}\sum_{i, j,k}(h_{ijk})^{2} \\
=&-\frac{3}{2}\lambda H|\nabla H|^{2}+\frac{3}{4}\lambda S(S-H^{2})(H-\lambda)-3\lambda\sum_{k}(h_{11}h^{2}_{23k} \\
 &+h_{22}h^{2}_{13k}+h_{33}h^{2}_{12k})+\frac{9}{2}\lambda Sh_{11}h_{22}h_{33}-\frac{3}{2}\lambda^{2}\sum_{k}(h_{22}h_{33}h^{2}_{1k}+h_{11}h_{33}h^{2}_{2k}\\
 &+h_{11}h_{22}h^{2}_{3k})+3\lambda\sum_{k}(h_{11}h_{22k}h_{33k}+h_{22}h_{11k}h_{33k}+h_{33}h_{11k}h_{22k}).
\endaligned
\end{equation}
\end{lemma}

\begin{proof}
\noindent By making use of the Ricci identities \eqref{2.1-7}, \eqref{2.1-9} and a direct calculation, we
can  obtain \eqref{2.1-17}.
Besides, from \eqref{2.1-16} in Lemma \ref{lemma 2.2}, we have
$$\sum_{i,j,k}h_{ijk}^{2}=-(S^{2}-\lambda f_{3}).$$
Then, by making use of the Ricci identities \eqref{2.1-7}, we obtain
\begin{equation*}
\aligned
 &-\frac{1}{2}\Delta_{-V} (S^{2}-\lambda f_{3})=\frac{1}{2}\lambda \Delta_{-V}f_{3}\\
=&-\frac{3}{2}\lambda H|\nabla H|^{2}+\frac{3}{4}\lambda S(S-H^{2})(H-\lambda)
 -\frac{9}{2}\lambda S(h_{11}h^{2}_{23}+h_{22}h^{2}_{13}+h_{33}h^{2}_{12}) \\
 &+\frac{3}{2}\lambda^{2}\sum_{k}(h^{2}_{23}h^{2}_{1k}+h^{2}_{13}h^{2}_{2k}+h^{2}_{12}h^{2}_{3k}+2h_{11}h_{23}h_{2k}h_{3k}+2h_{22}h_{13}h_{1k}h_{3k}\\
 &+2h_{33}h_{12}h_{1k}h_{2k})-3\lambda\sum_{k}(h_{11}h^{2}_{23k}+h_{22}h^{2}_{13k}+h_{33}h^{2}_{12k}+2h_{23}h_{23k}h_{11k} \\
 &+2h_{13}h_{13k}h_{22k}+2h_{12}h_{12k}h_{33k})+\frac{9}{2}\lambda Sh_{11}h_{22}h_{33}-\frac{3}{2}\lambda^{2}\sum_{k}(h_{22}h_{33}h^{2}_{1k}\\
 &+h_{11}h_{33}h^{2}_{2k}+h_{11}h_{22}h^{2}_{3k})+3\lambda\sum_{k}(h_{11}h_{22k}h_{33k}+h_{22}h_{11k}h_{33k}+h_{33}h_{11k}h_{22k})\\
 &+9\lambda Sh_{12}h_{13}h_{23}-3\lambda^{2}\sum_{k}(h_{12}h_{13}h_{2k}h_{3k}+h_{12}h_{23}h_{1k}h_{3k}+h_{13}h_{23}h_{1k}h_{2k})\\
 &+6\lambda\sum_{k}(h_{12}h_{13k}h_{23k}+h_{13}h_{12k}h_{23k}+h_{23}h_{12k}h_{13k}).
\endaligned
\end{equation*}
If diagonalized $(h_{ij})$ at some point, it is easy to get \eqref{2.1-19}.
\end{proof}

\begin{lemma}\label{lemma 2.4}
Let $x: M^{3}\rightarrow\mathbb{R}^{4}_{1}$ be an $3$-dimensional complete $\lambda$-translator in $\mathbb R^{4}_{1}$. Then we can choose a local field of orthonormal frames on $M^3$ such that, at the point,
$h_{ij}=\lambda_i\delta_{ij}$,

$$f_{3}=\frac{H}{2}(3S-H^{2})+3\lambda_{1}\lambda_{2}\lambda_{3},$$
$$f_{4}=\frac{4}{3}Hf_{3}-H^{2}S+\frac{1}{6}H^{4}+\frac{1}{2}S^{2},$$
$$f_{5}=\frac{5}{6}H^{2}f_{3}+\frac{5}{6}Sf_{3}-\frac{5}{6}H^{3}S+\frac{1}{6}H^{5},$$
$$
\nabla_{l}f_{3}=3\sum_{i,j,k}h_{ijl}h_{jk}h_{ki}, \ \ \text{for } \ l=1, 2, 3,
$$
$$
\nabla_{p}\nabla_{l}f_{3}=3\sum_{i,j,k}h_{ijlp}h_{jk}h_{ki}+6\sum_{i,j,k}h_{ijl}h_{jkp}h_{ki}, \ \ \text{for } \ l,p=1, 2, 3,
$$
and
$$
\nabla_{m}f_{4}=4\sum_{i,j,k,l}h_{ijm}h_{jk}h_{kl}h_{li}, \ \ \text{for } \ m=1, 2, 3,
$$

\begin{equation*}
\begin{aligned}
\nabla_{p}\nabla_{m}f_{4}=&4\sum_{i,j,k,l}h_{ijmp}h_{jk}h_{kl}h_{li}\\
                          &+4\sum_{i,j,k,l}h_{ijm}(2h_{jkp}h_{kl}h_{li}+h_{jk}h_{klp}h_{li}), \  \text{for } \ m,p=1, 2, 3.
\end{aligned}
\end{equation*}

\begin{equation}\label{2.1-20}
\nabla_{k}f_{4}=\frac{4}{3}f_{3} H_{,k}+\frac{4}{3}H\nabla_{k}f_{3}-2SHH_{,k}+\frac{2}{3}H^{3}H_{,k},
\end{equation}

\begin{equation}\label{2.1-21}
\begin{aligned}
\nabla_{l}\nabla_{k}f_{4} =&\frac{4}{3}f_{3}H_{,kl}-2SHH_{,kl}
   +\frac{2}{3}H^{3}H_{,kl}+\frac{4}{3}H\nabla_{l}\nabla_{k}f_{3}+\frac{4}{3}\nabla_{l}f_{3} H_{,k} \\
  &+\frac{4}{3}H_{,l}\nabla_{k}f_{3}-2SH_{,k}H_{,l}+2H^{2}H_{,k}H_{,l},
\end{aligned}
\end{equation}
for k, l=1, 2, 3.
\end{lemma}

\noindent To make use of the maximum principle
of Omori-Yau type, we prove the following lemma.

\begin{lemma}\label{lemma 2.5}
For a space-like complete $\lambda$-translator $x:M^{n}\rightarrow \mathbb{R}^{n+1}_{1}$ with the translating vector
$T$ and non-zero constant squared norm $S$ of the second fundamental form, the Bakry-Emery Ricci tensor $Ric_{-V}$ is bounded from below, where $V = T^{T}$.
\end{lemma}
\begin{proof}
Let $e$ be an arbitrary unit eigenvector of the symmetric two-tensor $Ricc_{-V}$. Choose
an orthonormal tangent frame field $\{{e_{i}}\}^{n}_{1}$ such that $e_{1} = e$. Then, by the definition of
$\lambda$-translator, we have
\begin{equation*}
\begin{aligned}
-\frac{1}{2}L_{-V}g(e,e)
=&\frac{1}{2}V(g(e_{1},e_{1}))-g([V,e_{1}],e_{1})\\
=&\frac{1}{2}\{g(\nabla_{V}e_{1},e_{1})+g(e_{1},\nabla_{V}e_{1})\}-g(\nabla_{V}e_{1}-\nabla_{e_{1}}V,e_{1})\\
=&g(\nabla_{e_{1}}(T-T^{\perp}),e_{1})\\
=&-g(\nabla_{e_{1}}T^{\perp},e_{1})\\
=&(H-\lambda)g(\nabla_{e_{1}}\vec{N},e_{1}),
\end{aligned}
\end{equation*}
and
\begin{equation*}
\begin{aligned}
g(\nabla_{e_{1}}\vec{N},e_{1})
=&g(d\vec{N}(e_{1}),e_{1})\\
=&g(\omega^{i}_{n+1}(e_{1})e_{i},e_{1})\\
=&g(\omega^{n+1}_{i}(e_{1})e_{i},e_{1})\\
=&h_{11}.
\end{aligned}
\end{equation*}
Therefore,
\begin{equation*}
\begin{aligned}
-\frac{1}{2}L_{-V}g(e,e)=&(H-\lambda)h_{11},\\
Ricc_{-V}(e,e)
=&Ricc(e,e)-\frac{1}{2}L_{-V}g(e,e)\\
=&-(Hh_{11}-\sum h^{2}_{1k})+(H-\lambda)h_{11}\\
=&\sum h^{2}_{1k}-\lambda h_{11}\\
\geq &\sum h^{2}_{1k}-\frac{1}{2}h^{2}_{11}-\frac{1}{2}\lambda^{2}\geq-\frac{1}{2}\lambda^{2}.
\end{aligned}
\end{equation*}
The proof of Lemma \ref{lemma 2.5} is finished.

\end{proof}
 \vskip10mm
\section{Proof of the main result}

\vskip2mm
\noindent
If $S=0$, we know that $x: M^{3}\to \mathbb{R}^{4}_{1}$ is $\mathbb{R}^{3}_{1}$, obviously. Next, we assume that $S>0$. From Lemma \ref{lemma 2.4}, it is sufficient to prove that $\inf H^{2}>0$.
We now prove the following theorems.

\begin{theorem}\label{theorem 3.1}
For a $3$-dimensional complete $\lambda$-translator $x:M^{3}\rightarrow \mathbb{R}^{4}_{1}$ with non-zero constant squared norm $S$ of the second fundamental form and constant $f_{4}$, then  $\inf H^{2}>0$,
where $S=\sum_{i,j}h_{ij}^2$ and $f_{4}=\sum_{i,j,k,l}h_{ij}h_{jk}h_{kl}h_{li}$.
\end{theorem}

\begin{proof}
If $\inf H^{2}=0$, there exists a sequence $\{p_{t}\}$ in $M^{3}$ such that
\begin{equation*}
\lim_{t\rightarrow\infty} H^{2}(p_{t})=\inf H^{2}=\bar H^2=0.
\end{equation*}

\noindent From \eqref{2.1-16},
\eqref{2.1-18}, \eqref{2.1-19} and $S$ being constant, we know that
$\{h_{ij}(p_{t})\}$,  $\{h_{ijk}(p_{t})\}$ and $\{h_{ijkl}(p_{t})\}$ are bounded sequences, one can assume
$$\lim_{t\rightarrow\infty}h_{ij}(p_{t})=\bar h_{ij}=\bar \lambda_i\delta_{ij}, \quad  \lim_{t\rightarrow\infty}h_{ijk}(p_{t})=\bar h_{ijk}, \quad \lim_{t\rightarrow\infty}h_{ijkl}(p_{t})=\bar h_{ijkl}$$
for $i, j, k, l=1, 2, 3$.
Then,
$$
\bar H=\sum_i \bar h_{ii}=\bar \lambda_{1}+\bar \lambda_{2}+\bar \lambda_{3}=0, \ S=\sum_{i,j}\bar h_{ij}^2=\bar \lambda^2_{1}+\bar \lambda^2_{2}+\bar \lambda^2_{3}=2(\bar \lambda^2_{1}+\bar \lambda^2_{2}+\bar \lambda_{1}\bar \lambda_{2}).
$$

\noindent From
$$H_{,i}=\sum_{k}h_{ik}\langle T, e_{k}\rangle, \ \ \text{\rm for} \ \ i=1, 2, 3,$$
we have
\begin{equation}\label{3.1-1}
\bar h_{11k}+\bar h_{22k}+\bar h_{33k}=\bar \lambda_{k}\lim_{t\rightarrow\infty} \langle T, e_{k} \rangle(p_{t}),\ \ \text{\rm for} \ \ k=1, 2, 3.
\end{equation}
Since
\begin{equation}\label{3.1-2}
\aligned
\nabla_{j}\nabla_{i}H
 =&\sum_{k}h_{ijk}\langle T,e_{k}\rangle+(H-\lambda)\sum_{k}h_{ik}h_{kj},
\endaligned
\end{equation}
we conclude
\begin{equation*}
\bar H_{,ij}=\sum_{k}\bar h_{ijk}\lim_{t\rightarrow\infty} \langle T,e_{k} \rangle(p_{t})+(\bar H-\lambda)\bar \lambda_{i}\bar \lambda_{j}\delta_{ij},
\end{equation*}
that is,
\begin{equation}\label{3.1-3}
\begin{cases}
\begin{aligned}
&\bar h_{1111}+\bar h_{2211}+\bar h_{3311}=\sum_{k}\bar h_{11k}\lim_{t\rightarrow\infty} \langle T,e_{k} \rangle(p_{t})-\lambda\bar \lambda^{2}_{1},\\
&\bar h_{1122}+\bar h_{2222}+\bar h_{3322}=\sum_{k}\bar h_{22k}\lim_{t\rightarrow\infty} \langle T,e_{k} \rangle(p_{t})-\lambda\bar \lambda^{2}_{2},\\
&\bar h_{1133}+\bar h_{2233}+\bar h_{3333}=\sum_{k}\bar h_{33k}\lim_{t\rightarrow\infty} \langle T,e_{k} \rangle(p_{t})-\lambda\bar \lambda^{2}_{3},\\
&\bar h_{1112}+\bar h_{2212}+\bar h_{3312}=\sum_{k}\bar h_{12k}\lim_{t\rightarrow\infty} \langle T,e_{k} \rangle(p_{t}),\\
&\bar h_{1113}+\bar h_{2213}+\bar h_{3313}=\sum_{k}\bar h_{13k}\lim_{t\rightarrow\infty} \langle T,e_{k} \rangle(p_{t}),\\
&\bar h_{1123}+\bar h_{2223}+\bar h_{3323}=\sum_{k}\bar h_{23k}\lim_{t\rightarrow\infty} \langle T,e_{k} \rangle(p_{t}).
\end{aligned}
\end{cases}
\end{equation}
Since $S$ is constant, we know
$$
\sum_{i,j}h_{ij}h_{ijk}=0, \ \ \text{for } \ k=1, 2, 3.
$$
Thus,
$$
\sum_{i,j}\bar h_{ij}\bar h_{ijk}=0, \ \ \text{for } \ k=1, 2, 3.
$$
Specifically,
\begin{equation}\label{3.1-4}
\bar\lambda_{1}\bar h_{11k}+\bar\lambda_{2}\bar h_{22k}+\bar\lambda_{3}\bar h_{33k}=0, \ \  \text{for } \ k=1, 2, 3.
\end{equation}

\noindent Now we consider three scenarios.

\vskip2mm
\noindent {\bf 1. The principal curvature $\bar \lambda_1$, $\bar \lambda_2$ and $\bar \lambda_3$ are all equal.}

\noindent  From $\bar H=\bar \lambda_1+\bar \lambda_2+\bar \lambda_3=0$, $\bar \lambda_1=\bar \lambda_2=\bar \lambda_3=0$, we get $S=0$. It is impossible since $S>0$.

\vskip2mm
\noindent {\bf 2. Two of the values of the principal curvature $\bar \lambda_1$, $\bar \lambda_2$ and $\bar \lambda_3$  are equal.}

\noindent Without loss of generality, we assume that $\bar \lambda_1=\bar \lambda_2\neq \bar \lambda_3$.

\noindent From $\bar H=\bar \lambda_1+\bar \lambda_2+\bar \lambda_3=0$, we infer that $\bar \lambda_{1}=\bar \lambda_{2}\neq 0$ and $\bar\lambda_3\neq0$.

\noindent From \eqref{2.1-20} in  Lemma \ref{lemma 2.4},  we obtain
$$\lim_{t\rightarrow\infty}f_{3}(p_{t})\neq 0,  \ \  \ \  \bar H_{,k}=0 \ \ \text{\rm for} \ \ k=1, 2, 3.$$

\noindent By \eqref{3.1-1} and $\bar H_{,k}=0$ for $k=1, 2, 3$, we have
$$
\lim_{t\rightarrow\infty} \langle T, e_{k} \rangle(p_{t})=0,\ \ \text{\rm for} \ \ k=1, 2, 3.
$$

\noindent From \eqref{3.1-3}, we have that

\begin{equation}\label{3.1-5}
\begin{cases}
\begin{aligned}
&\bar H_{,11}=\bar h_{1111}+\bar h_{2211}+\bar h_{3311}=-\lambda\bar \lambda^{2}_{1},\\
&\bar H_{,22}=\bar h_{1122}+\bar h_{2222}+\bar h_{3322}=-\lambda\bar \lambda^{2}_{2},\\
&\bar H_{,33}=\bar h_{1133}+\bar h_{2233}+\bar h_{3333}=-\lambda\bar \lambda^{2}_{3}.
\end{aligned}
\end{cases}
\end{equation}

\noindent From \eqref{2.1-21} in Lemma \ref{lemma 2.4} and $\bar H_{,k}=0$ for $k=1, 2, 3$, we have
\begin{equation}\label{3.1-6}
\lim_{t\rightarrow\infty}\nabla_{l}\nabla_{k}f_{4}(p_{t})=0,\ \ \bar H_{,kl}=0,\ \ \text{\rm for} \ \ k, l=1, 2, 3.
\end{equation}
Then, it follows from \eqref{3.1-5} that $\bar H_{,kk}=-\lambda\bar \lambda^{2}_{k}=0$ for $k=1, 2, 3$. It is a contradiction.

\vskip2mm
\noindent {\bf 3. The values of the principal curvature $\bar \lambda_1$, $\bar \lambda_{2}$ and $\bar \lambda_3$ are not equal to each other.}

\noindent {\bf Case 1: $\bar \lambda_1\bar \lambda_2\bar \lambda_{3}=0$}.

\noindent Without loss of generality, we assume that $\bar \lambda_{3}=0$. That is, $\bar \lambda_{1}\neq 0$, $\bar \lambda_{2}\neq 0$ and $\bar \lambda_{1}\neq\bar \lambda_{2}$.

\noindent From $\bar H=0$, $\bar \lambda_{3}=0$ and $f_{3}=\frac{H}{2}(3S-H^{2})+3\lambda_{1}\lambda_{2}\lambda_{3}$ , we have $$\lim_{t\rightarrow\infty}f_{3}(p_{t})=0.$$

\noindent From \eqref{2.1-16} in Lemma \ref{lemma 2.2}, we have
\begin{equation*}
\sum_{i,j,k}h_{ijk}^2+S^{2}-\lambda f_{3}=0.
\end{equation*}

\noindent Since $\lim\limits_{t\rightarrow\infty}f_{3}(p_{t})=0$,
we have
$$\sum_{i,j,k}\bar h_{ijk}^2+S^{2}=0,$$
and then, $$S=0.$$ It is a contradiction.

\noindent {\bf Case 2: $\bar \lambda_{1}\bar \lambda_{2} \bar \lambda_{3}\neq0$}.

\noindent From $S>0$, $\bar H=0$ and $f_{3}=\frac{H}{2}(3S-H^{2})+3\lambda_{1}\lambda_{2}\lambda_{3}$, we have
$$\lim\limits_{t\rightarrow\infty}f_{3}(p_{t})\neq 0.$$

\noindent Since $f_{4}=\frac{4}{3}Hf_{3}-H^{2}S+\frac{1}{6}H^{4}+\frac{1}{2}S^{2}$ and $\lim\limits_{t\rightarrow\infty}f_{3}(p_{t})\neq 0$, we get
\begin{equation*}
\begin{aligned}
0=\nabla_{k}f_{4}
 =&\frac{4}{3}f_{3} H_{,k}+\frac{4}{3}H\nabla_{k}f_{3}-2SHH_{,k}+\frac{2}{3}H^{3}H_{,k}, \\
0=\nabla_{l}\nabla_{k}f_{4}
 =&\frac{4}{3}f_{3}H_{,kl}-2SHH_{,kl}
   +\frac{2}{3}H^{3}H_{,kl}+\frac{4}{3}H\nabla_{l}\nabla_{k}f_{3}+\frac{4}{3}\nabla_{l}f_{3} H_{,k} \\
  &+\frac{4}{3}H_{,l}\nabla_{k}f_{3}-2SH_{,k}H_{,l}+2H^{2}H_{,k}H_{,l},
\ \ \text{for } \ k, l=1, 2, 3.
\end{aligned}
\end{equation*}
Then, $\bar H_{,k}=0$ and $\bar H_{,kl}=0$ for $k, l=1, 2, 3$.

\noindent Especially,
\begin{equation}\label{3.1-7}
\begin{cases}
\begin{aligned}
&\bar H_{,1}=\bar \lambda_{1}\lim_{t\rightarrow\infty}\langle T,e_{1} \rangle(p_{t})=0,\\
&\bar H_{,2}=\bar \lambda_{2}\lim_{t\rightarrow\infty}\langle T,e_{2} \rangle(p_{t})=0,\\
&\bar H_{,3}=\bar \lambda_{3}\lim_{t\rightarrow\infty}\langle T,e_{3} \rangle(p_{t})=0,
\end{aligned}
\end{cases}
\end{equation}
and
\begin{equation}\label{3.1-8}
\begin{cases}
\begin{aligned}
&\bar H_{,11}=\sum_{k}h_{11k}\lim_{t\rightarrow\infty}\langle T,e_{k} \rangle(p_{t})-\lambda\bar \lambda^{2}_{1}=0,\\
&\bar H_{,22}=\sum_{k}h_{22k}\lim_{t\rightarrow\infty}\langle T,e_{k} \rangle(p_{t})-\lambda\bar \lambda^{2}_{2}=0,\\
&\bar H_{,33}=\sum_{k}h_{33k}\lim_{t\rightarrow\infty}\langle T,e_{k} \rangle(p_{t})-\lambda\bar \lambda^{2}_{3}=0.
\end{aligned}
\end{cases}
\end{equation}
\noindent From \eqref{3.1-7} and $\bar \lambda_{k} \neq 0$ for $k=1,2,3$, one has
\begin{equation*}
\lim_{t\rightarrow\infty}\langle T,e_{k} \rangle(p_{t})=0, \ \  \text{for} \ k=1, 2, 3.
\end{equation*}
\noindent By \eqref{3.1-8}, we know that $\bar \lambda_{k}=0$ for $k=1, 2, 3$.
It is a contradiction.
\end{proof}

\begin{theorem}\label{theorem 3.2}
For a $3$-dimensional complete $\lambda$-translator $x:M^{3}\rightarrow \mathbb{R}^{4}_{1}$ with non-zero constant squared norm $S$ of the second fundamental form and constant $f_{4}$, where $S=\sum_{i,j}h_{ij}^2$ and $f_{4}=\sum_{i,j,k,l}h_{ij}h_{jk}h_{kl}h_{li}$,
we have either
\begin{enumerate}
\item $\lambda^{2}=S$ and $\sup H^{2}=S$, or
\item $\lambda^{2}=2S$ and $\sup H^{2}=2S$, or
\item $\lambda^{2}=3S$ and $\sup H^{2}=3S$.
\end{enumerate}
\end{theorem}
\begin{proof}
\noindent From Lemma \ref{lemma 2.2}, we have
\begin{equation*}
\frac{1}{2}\Delta_{-V}H^{2}=|\nabla H|^{2}+S(H-\lambda)H.
\end{equation*}
At each point $p\in M^{3}$, we choose $e_{1}$, $e_{2}$ and $e_{3}$ such that
$$
h_{ij}=\lambda_i\delta_{ij}.
$$
\noindent From $2ab\leq \alpha a^{2}+\dfrac1{\alpha}b^{2}$, we obtain
$$
S=\lambda_{1}^{2} +\lambda_{2}^{2} +\lambda_{3}^{2}, \ \ H^{2}=(\lambda_{1}+\lambda_{2}+\lambda_{3})^{2}\leq 3(\lambda_{1}^{2} +\lambda_{2}^{2} +\lambda_{3}^{2})=3S.
$$
Hence, we have on $M^{3}$
$$
H^{2}\leq 3S
$$
and the equality holds if and only if $\lambda_{1}=\lambda_{1}=\lambda_{3}$.

\noindent From Lemma \ref{lemma 2.5},
we know that the Bakry-Emery Ricci tensor $Ric_{-V}$ of $x: M^{3}\to \mathbb{R}^{4}_{1}$ is bounded from below.
We can apply the generalized maximum principle for the operator $\Delta_{-V}$
to the function $H^{2}$. Thus, there exists a sequence $\{p_{t}\}$ in $M^{3}$ such that
\begin{equation*}
\lim_{t\rightarrow\infty} H^{2}(p_{t})=\sup H^{2},\quad
\lim_{t\rightarrow\infty} |\nabla H^{2}(p_{t})|=0,\quad
\lim_{t\rightarrow\infty}\Delta_{-V}H^{2}(p_{t})\leq 0.
\end{equation*}

\noindent For $S\neq 0$, from Theorem \ref{theorem 3.1}, we know that $\sup H^{2}\geq \inf H^{2}>0$.
Without loss of the generality, at each point $p_{t}$, we can assume $H(p_{t})\neq 0$. From \eqref{2.1-16},
\eqref{2.1-17}, \eqref{2.1-18} and $S=constant$, we know that
$\{h_{ij}(p_{t})\}$,  $\{h_{ijk}(p_{t})\}$ and $\{h_{ijkl}(p_{t})\}$  are bounded sequences for $ i, j, k, l = 1,2,3$.
We can assume
\begin{equation*}
\begin{aligned}
&\lim_{t\rightarrow\infty}f_{3}(p_{t})=\bar f_{3}, \ \ \lim_{t\rightarrow\infty}f_{5}(p_{t})=\bar f_{5}, \ \ \lim_{t\rightarrow\infty}h_{ij}(p_{t})=\bar h_{ij}=\bar \lambda_i\delta_{ij}, \\
&\lim_{t\rightarrow\infty}h_{ijk}(p_{t})=\bar h_{ijk}, \ \  \lim_{t\rightarrow\infty}h_{ijkl}(p_{t})=\bar h_{ijkl}, \ \ \text{\rm for} \ i, j, k=1, 2, 3.
\end{aligned}
\end{equation*}

\vskip1mm
\noindent
From Lemma \ref{lemma 2.2}, we get
\begin{equation*}
\begin{cases}
\begin{aligned}
&\lim_{t\rightarrow\infty} H^{2}(p_{t})=\sup H^{2}=\bar H^{2},\quad
\lim_{t\rightarrow\infty} |\nabla H^{2}(p_{t})|=0,\\
&0\geq
\lim_{t\rightarrow\infty} |\nabla H|^{2}(p_{t})+S(\bar H-\lambda)\bar H.
\end{aligned}
\end{cases}
\end{equation*}

\noindent From $\lim_{t\rightarrow\infty} |\nabla H^{2}(p_{t})|=0$ and $|\nabla H^{2}|^{2}=4\sum_{k}(HH_{,k})^{2}$, we have
\begin{equation}\label{3.1-9}
\bar H_{,k}=0, \ \ \text{\rm for} \  k=1, 2, 3,
\end{equation}
that is,
\begin{equation}\label{3.1-10}
\bar h_{11k}+\bar h_{22k}+\bar h_{33k}=0, \ \ \text{\rm for} \  k=1, 2, 3.
\end{equation}

\noindent Since $x$ is a $\lambda$-translator, from \eqref{2.1-13}, we have
$$H_{,i}=\sum_{k}h_{ik}\langle T, e_{k}\rangle, \ \ \text{\rm for} \ \ i=1, 2, 3,$$
$$\nabla_{j}\nabla_{i}H=\sum_{k}h_{ijk}\langle T,e_{k}\rangle+(H-\lambda)\sum_{k}h_{ik}h_{kj},
\ \ \text{\rm for} \ \ i,j=1, 2, 3.
$$
Thus,
$$\bar H_{,i}=\bar h_{11i}+\bar h_{22i}+\bar h_{33i}=\bar \lambda_{i}\lim_{t\rightarrow\infty} \langle T, e_{i} \rangle(p_{t}),\ \ \text{\rm for} \ \ i=1, 2, 3,$$

$$\bar H_{,ij}=\sum_{k}\bar h_{ijk}\lim_{t\rightarrow\infty} \langle T,e_{k} \rangle(p_{t})+(\bar H-\lambda)\bar \lambda_{i}\bar \lambda_{j}\delta_{ij},\ \ \text{\rm for} \ \ i,j=1, 2, 3.$$

Especially,
\begin{equation}\label{3.1-11}
\begin{cases}
\begin{aligned}
&\bar H_{,1}=\bar \lambda_{1} \lim_{t\rightarrow\infty} \langle T,e_{1} \rangle(p_{t})=0,\\
&\bar H_{,2}=\bar \lambda_{2} \lim_{t\rightarrow\infty} \langle T,e_{2} \rangle(p_{t})=0,\\
&\bar H_{,3}=\bar \lambda_{3} \lim_{t\rightarrow\infty} \langle T,e_{3} \rangle(p_{t})=0,
\end{aligned}
\end{cases}
\end{equation}
and
\begin{equation}\label{3.1-12}
\begin{cases}
\begin{aligned}
&\bar h_{1111}+\bar h_{2211}+\bar h_{3311}=\sum_{k}\bar h_{11k}\lim_{t\rightarrow\infty} \langle T,e_{k} \rangle(p_{t})+(\bar H-\lambda)\bar \lambda^{2}_{1},\\
&\bar h_{1122}+\bar h_{2222}+\bar h_{3322}=\sum_{k}\bar h_{22k}\lim_{t\rightarrow\infty} \langle T,e_{k} \rangle(p_{t})+(\bar H-\lambda)\bar \lambda^{2}_{2},\\
&\bar h_{1133}+\bar h_{2233}+\bar h_{3333}=\sum_{k}\bar h_{33k}\lim_{t\rightarrow\infty} \langle T,e_{k} \rangle(p_{t})+(\bar H-\lambda)\bar \lambda^{2}_{3},\\
&\bar h_{1112}+\bar h_{2212}+\bar h_{3312}=\sum_{k}\bar h_{12k}\lim_{t\rightarrow\infty} \langle T,e_{k} \rangle(p_{t}),\\
&\bar h_{1113}+\bar h_{2213}+\bar h_{3313}=\sum_{k}\bar h_{13k}\lim_{t\rightarrow\infty} \langle T,e_{k} \rangle(p_{t}),\\
&\bar h_{1123}+\bar h_{2223}+\bar h_{3323}=\sum_{k}\bar h_{23k}\lim_{t\rightarrow\infty} \langle T,e_{k} \rangle(p_{t}).
\end{aligned}
\end{cases}
\end{equation}

\noindent Since $S$ is constant, we know
$$
\sum_{i,j}h_{ij}h_{ijk}=0, \ \ \text{for } \ k=1, 2, 3,
$$
$$
\sum_{i,j}h_{ij}h_{ijkl}+\sum_{i,j}h_{ijk}h_{ijl}=0, \ \ \text{for } \ k,l=1, 2, 3.
$$
Thus,
$$
\sum_{i,j}\bar h_{ij}\bar h_{ijk}=0, \ \ \text{for } \ k=1, 2, 3,
$$
$$
\sum_{i,j}\bar h_{ij}\bar h_{ijkl}+\sum_{i,j}\bar h_{ijk}\bar h_{ijl}=0, \ \ \text{for } \ k,l=1, 2, 3.
$$
Specifically,
\begin{equation}\label{3.1-13}
\bar\lambda_{1}\bar h_{11k}+\bar\lambda_{2}\bar h_{22k}+\bar\lambda_{3}\bar h_{33k}=0, \ \  \text{for } \ k=1, 2, 3,
\end{equation}
\begin{equation}\label{3.1-14}
\begin{cases}
\begin{aligned}
\bar \lambda_{1}\bar h_{1111}+\bar \lambda_{2}\bar h_{2211}+\bar \lambda_{3}\bar h_{3311}
=&-\bar h^{2}_{111}-\bar h^{2}_{221}-\bar h^{2}_{331}-2\bar h^{2}_{121}\\
   &-2\bar h^{2}_{131}-2\bar h^{2}_{231},\\
\bar \lambda_{1}\bar h_{1122}+\bar \lambda_{2}\bar h_{2222}+\bar \lambda_{3}\bar h_{3322}
=&-\bar h^{2}_{112}-\bar h^{2}_{222}-\bar h^{2}_{332}-2\bar h^{2}_{122}\\
   &-2\bar h^{2}_{132}-2\bar h^{2}_{232},\\
\bar \lambda_{1}\bar h_{1133}+\bar \lambda_{2}\bar h_{2233}+\bar \lambda_{3}\bar h_{3333}
=&-\bar h^{2}_{113}-\bar h^{2}_{223}-\bar h^{2}_{333}-2\bar h^{2}_{123}\\
  &-2\bar h^{2}_{133}-2\bar h^{2}_{233},\\
\bar \lambda_{1}\bar h_{1112}+\bar \lambda_{2}\bar h_{2212}+\bar \lambda_{3}\bar h_{3312}
=&-\bar h_{111}\bar h_{112}-\bar h_{221}\bar h_{222}-\bar h_{331}\bar h_{332}\\
    &-2\bar h_{121}\bar h_{122}-2\bar h_{131}\bar h_{132}-2\bar h_{231}\bar h_{232},\\
\bar \lambda_{1}\bar h_{1113}+\bar \lambda_{2}\bar h_{2213}+\bar \lambda_{3}\bar h_{3313}
=&-\bar h_{111}\bar h_{113}-\bar h_{221}\bar h_{223}-\bar h_{331}\bar h_{333}\\
   &-2\bar h_{121}\bar h_{123}-2\bar h_{131}\bar h_{133}-2\bar h_{231}\bar h_{233},\\
\bar \lambda_{1}\bar h_{1123}+\bar \lambda_{2}\bar h_{2223}+\bar \lambda_{3}\bar h_{3323}
=&-\bar h_{112}\bar h_{113}-\bar h_{222}\bar h_{223}-\bar h_{332}\bar h_{333}\\
 &-2\bar h_{122}\bar h_{123}-2\bar h_{132}\bar h_{133}-2\bar h_{232}\bar h_{233}.
\end{aligned}
\end{cases}
\end{equation}
From Ricci identities \eqref{2.1-7}, we obtain
\begin{equation*}
\bar h_{ijkl}-\bar h_{ijlk}=-(\bar\lambda_{i}\bar\lambda_{j}\bar\lambda_{k}\delta_{il}\delta_{jk}-\bar\lambda_{i}\bar\lambda_{j}\bar\lambda_{l}\delta_{ik}\delta_{jl}
+\bar\lambda_{i}\bar\lambda_{j}\bar\lambda_{k}\delta_{ik}\delta_{jl}-\bar\lambda_{i}\bar\lambda_{j}\bar\lambda_{l}\delta_{il}\delta_{jk}),
\end{equation*}
that is,
\begin{equation}\label{3.1-15}
\begin{cases}
\begin{aligned}
& \bar h_{1212}-\bar h_{1221}=-\bar \lambda_{1}\bar \lambda_{2}(\bar \lambda_{1}-\bar \lambda_{2}),\ \ \bar h_{1313}-\bar h_{1331}=-\bar \lambda_{1}\bar \lambda_{3}(\bar \lambda_{1}-\bar \lambda_{3}),\\
&\bar h_{2323}-\bar h_{2332}=-\bar \lambda_{2}\bar \lambda_{3}(\bar \lambda_{2}-\bar \lambda_{3}), \ \ \bar h_{iikl}-\bar h_{iilk}=0, \ \ \text{for } \ i,k,l=1, 2, 3.
\end{aligned}
\end{cases}
\end{equation}
Since $f_{4}$ is constant, we know from the Lemma \ref{lemma 2.4},
$$
0=\nabla_{m}f_{4}=4\sum_{i,j,k,l}h_{ijm}h_{jk}h_{kl}h_{li},
$$
$$
0=\nabla_{p}\nabla_{m}f_{4}=4\sum_{i,j,k,l}h_{ijmp}h_{jk}h_{kl}h_{li}+4\sum_{i,j,k,l}h_{ijm}(2h_{jkp}h_{kl}h_{li}+h_{jk}h_{klp}h_{li}),
$$
for $m,p=1, 2, 3$.
Thus,

\begin{equation}\label{3.1-16}
\bar\lambda^{3}_{1}\bar h_{11k}+\bar\lambda^{3}_{2}\bar h_{22k}+\bar\lambda^{3}_{3}\bar h_{33k}=0, \ \ \text{for } \ k=1, 2, 3,
\end{equation}

\begin{equation}\label{3.1-17}
\begin{cases}
\begin{aligned}
&\bar \lambda^{3}_{1}\bar h_{1111}+\bar \lambda^{3}_{2}\bar h_{2211}+\bar \lambda^{3}_{3}\bar h_{3311}\\
=&-3\bar \lambda^{2}_{1}\bar h^{2}_{111}-3\bar \lambda^{2}_{2}\bar h^{2}_{221}-3\bar \lambda^{2}_{3}
\bar h^{2}_{331}-2(\bar \lambda^{2}_{1}+\bar \lambda^{2}_{2}+\bar \lambda_{1}\bar \lambda_{2})\bar h^{2}_{121} \\
&-2(\bar\lambda^{2}_{1}+\bar\lambda^{2}_{3}+\bar\lambda_{1}\bar\lambda_{3})\bar h^{2}_{131}
-2(\bar\lambda^{2}_{2}+\bar\lambda^{2}_{3}+\bar\lambda_{2}\bar\lambda_{3})\bar h^{2}_{231},\\

&\bar \lambda^{3}_{1}\bar h_{1122}+\bar \lambda^{3}_{2}\bar h_{2222}+\bar \lambda^{3}_{3}\bar h_{3322}\\
=&-3\bar \lambda^{2}_{1}\bar h^{2}_{112}-3\bar \lambda^{2}_{2}\bar h^{2}_{222}-3\bar \lambda^{2}_{3}
\bar h^{2}_{332}-2(\bar \lambda^{2}_{1}+\bar \lambda^{2}_{2}+\bar \lambda_{1}\bar \lambda_{2})\bar h^{2}_{122} \\
&-2(\bar\lambda^{2}_{1}+\bar\lambda^{2}_{3}+\bar\lambda_{1}\bar\lambda_{3})\bar h^{2}_{132}
-2(\bar\lambda^{2}_{2}+\bar\lambda^{2}_{3}+\bar\lambda_{2}\bar\lambda_{3})\bar h^{2}_{232},\\

&\bar \lambda^{3}_{1}\bar h_{1133}+\bar \lambda^{3}_{2}\bar h_{2233}+\bar \lambda^{3}_{3}\bar h_{3333}\\
=&-3\bar \lambda^{2}_{1}\bar h^{2}_{113}-3\bar \lambda^{2}_{2}\bar h^{2}_{223}-3\bar \lambda^{2}_{3}
\bar h^{2}_{333}-2(\bar \lambda^{2}_{1}+\bar \lambda^{2}_{2}+\bar \lambda_{1}\bar \lambda_{2})\bar h^{2}_{123} \\
&-2(\bar\lambda^{2}_{1}+\bar\lambda^{2}_{3}+\bar\lambda_{1}\bar\lambda_{3})\bar h^{2}_{133}
-2(\bar\lambda^{2}_{2}+\bar\lambda^{2}_{3}+\bar\lambda_{2}\bar\lambda_{3})\bar h^{2}_{233},\\

&\bar \lambda^{3}_{1}\bar h_{1112}+\bar \lambda^{3}_{2}\bar h_{2212}+\bar \lambda^{3}_{3}\bar h_{3312}\\
=&-3\bar \lambda^{2}_{1}\bar h_{111}\bar h_{112}-3\bar \lambda^{2}_{2}\bar h_{221}\bar h_{222}-3\bar \lambda^{2}_{3}
\bar h_{331}\bar h_{332}-2(\bar \lambda^{2}_{1}+\bar \lambda^{2}_{2} \\
&+\bar \lambda_{1}\bar \lambda_{2})\bar h_{121}\bar h_{122}-2(\bar\lambda^{2}_{1}+\bar\lambda^{2}_{3}
+\bar\lambda_{1}\bar\lambda_{3})\bar h_{131}\bar h_{132}-2(\bar\lambda^{2}_{2}+\bar\lambda^{2}_{3} \\
&+\bar\lambda_{2}\bar\lambda_{3})\bar h_{231}\bar h_{232},\\

&\bar \lambda^{3}_{1}\bar h_{1113}+\bar \lambda^{3}_{2}\bar h_{2213}+\bar \lambda^{3}_{3}\bar h_{3313}\\
=&-3\bar \lambda^{2}_{1}\bar h_{111}\bar h_{113}-3\bar \lambda^{2}_{2}\bar h_{221}\bar h_{223}-3\bar \lambda^{2}_{3}
\bar h_{331}\bar h_{333}-2(\bar \lambda^{2}_{1}+\bar \lambda^{2}_{2} \\
&+\bar \lambda_{1}\bar \lambda_{2})\bar h_{121}\bar h_{123}-2(\bar\lambda^{2}_{1}+\bar\lambda^{2}_{3}
+\bar\lambda_{1}\bar\lambda_{3})\bar h_{131}\bar h_{133}-2(\bar\lambda^{2}_{2}+\bar\lambda^{2}_{3} \\
&+\bar\lambda_{2}\bar\lambda_{3})\bar h_{231}\bar h_{233},\\

&\bar \lambda^{3}_{1}\bar h_{1123}+\bar \lambda^{3}_{2}\bar h_{2223}+\bar \lambda^{3}_{3}\bar h_{3323}\\
=&-3\bar \lambda^{2}_{1}\bar h_{112}\bar h_{113}-3\bar \lambda^{2}_{2}\bar h_{222}\bar h_{223}-3\bar \lambda^{2}_{3}
\bar h_{332}\bar h_{333}-2(\bar \lambda^{2}_{1}+\bar \lambda^{2}_{2} \\
&+\bar \lambda_{1}\bar \lambda_{2})\bar h_{122}\bar h_{123}-2(\bar\lambda^{2}_{1}+\bar\lambda^{2}_{3}
+\bar\lambda_{1}\bar\lambda_{3})\bar h_{132}\bar h_{133}-2(\bar\lambda^{2}_{2}+\bar\lambda^{2}_{3} \\
&+\bar\lambda_{2}\bar\lambda_{3})\bar h_{232}\bar h_{233}.
\end{aligned}
\end{cases}
\end{equation}

\noindent Now we consider three scenarios.

\vskip2mm
\noindent {\bf 1. The principal curvature $\bar \lambda_1$, $\bar \lambda_2$ and $\bar \lambda_3$ are all equal.}

\noindent From $\bar H=\bar \lambda_{1}+\bar \lambda_{2}+\bar \lambda_{3} \neq 0$, $\bar \lambda_{1}=\bar \lambda_{2}=\bar \lambda_{3}\neq 0$, we get
$$\bar H^{2}=3S.$$

\noindent From \eqref{3.1-11} and $\bar \lambda_{k} \neq 0$ for $k=1, 2, 3$,
we have
\begin{equation}\label{3.1-18}
\bar H_{,k}=\bar \lambda_{k}\lim_{t\rightarrow\infty} \langle T, e_{k} \rangle(p_{t})=0, \ \
\lim_{t\rightarrow\infty} \langle T,e_{k}\rangle=0, \ \ \text{for } \ k=1, 2, 3.
\end{equation}
\noindent From \eqref{3.1-14}, \eqref{3.1-17} and $\bar \lambda_{1}=\bar \lambda_{2}=\bar \lambda_{3}$, we have
\begin{equation}\label{3.1-19}
\begin{cases}
\begin{aligned}
\bar \lambda_{1}\sum_{i}h_{ii11}
=&-(\bar h^{2}_{111}+\bar h^{2}_{221}+\bar h^{2}_{331})-2(\bar h^{2}_{121}+\bar h^{2}_{131}+\bar h^{2}_{231}),\\
\bar \lambda_{1}\sum_{i}h_{ii22}
=&-(\bar h^{2}_{112}+\bar h^{2}_{222}+\bar h^{2}_{332})-2(\bar h^{2}_{122}+\bar h^{2}_{132}+\bar h^{2}_{232}),\\
\bar \lambda_{1}\sum_{i}h_{ii33}
=&-(\bar h^{2}_{113}+\bar h^{2}_{223}+\bar h^{2}_{333})-2(\bar h^{2}_{123}+\bar h^{2}_{133}+\bar h^{2}_{233}),
\end{aligned}
\end{cases}
\end{equation}
and
\begin{equation}\label{3.1-20}
\begin{cases}
\begin{aligned}
&\bar \lambda^{3}_{1}\sum_{i}h_{ii11}
=-3\bar \lambda^{2}_{1}(\bar h^{2}_{111}+\bar h^{2}_{221}+\bar h^{2}_{331})-6\bar \lambda^{2}_{1}(\bar h^{2}_{121}+\bar h^{2}_{131}+\bar h^{2}_{231}),\\

&\bar \lambda^{3}_{1}\sum_{i}h_{ii22}
=-3\bar \lambda^{2}_{1}(\bar h^{2}_{112}+\bar h^{2}_{222}+\bar h^{2}_{332})-6\bar \lambda^{2}_{1}(\bar h^{2}_{122}+\bar h^{2}_{132}+\bar h^{2}_{232}),\\

&\bar \lambda^{3}_{1}\sum_{i}h_{ii33}
=-3\bar \lambda^{2}_{1}(\bar h^{2}_{113}+\bar h^{2}_{223}+\bar h^{2}_{333})-6\bar \lambda^{2}_{1}(\bar h^{2}_{123}+\bar h^{2}_{133}+\bar h^{2}_{233}).
\end{aligned}
\end{cases}
\end{equation}

\noindent From \eqref{3.1-19} and \eqref{3.1-20}, we have
\begin{equation}\label{3.1-21}
\bar h_{ijk}=0, \ \ \bar H_{,kk}=0, \ \ \text{for } \ i,j,k=1, 2, 3.
\end{equation}
\noindent From \eqref{3.1-12}, \eqref{3.1-18} and \eqref{3.1-21}, we have
$$\lambda^{2}=\bar H^{2}=\sup H^{2}, \ \ \lambda^{2}=3S.$$

\vskip2mm
\noindent {\bf 2. Two of the values of the principal curvature $\bar \lambda_1$, $\bar \lambda_2$ and $\bar \lambda_3$  are equal.}

\noindent Without loss of generality, we assume that $\bar \lambda_{1}\neq \bar \lambda_{2}=\bar \lambda_{3}$, and then, $$\bar H =\bar \lambda_{1}+2\bar \lambda_{2}\neq 0.$$

\noindent From \eqref{3.1-10} and \eqref{3.1-13}, we get
\begin{equation}\label{3.1-22}
\bar h_{11k}=0, \ \ \bar h_{22k}+\bar h_{33k}=0,\ \ \text{for } \ k=1, 2, 3.
\end{equation}
\noindent {\bf Case 1: $\bar \lambda_{1}\bar \lambda_{2}\bar \lambda_{3}=0$.}

\noindent {\bf Subcase 1.1: $\bar \lambda_{1}\neq 0$, $\bar \lambda_{2}=\bar \lambda_{3}=0$.}

\noindent Since $\bar H^2\neq 0$, we have that $\bar H^{2} = S$ and $\bar f_{3}= \bar\lambda_{1}S$.

\noindent From the first equation in \eqref{3.1-17} and \eqref{3.1-22}, we have
$$\bar h_{1111}=0,$$
and then, by \eqref{3.1-14}, we know
\begin{equation*}
\begin{aligned}
0=&\bar \lambda_{1}\bar h_{1111}+\bar \lambda_{2}\bar h_{2211}+\bar \lambda_{3}\bar h_{3311}\\
 =&-\bar h^{2}_{111}-\bar h^{2}_{221}-\bar h^{2}_{331}-2\bar h^{2}_{121}-2\bar h^{2}_{131}-2\bar h^{2}_{231},\\
 =&-\bar h^{2}_{221}-\bar h^{2}_{331}-2\bar h^{2}_{231},
\end{aligned}
\end{equation*}
where $\bar \lambda_{2}=\bar \lambda_{3}=0$.
Thus, $$\bar h_{221}=\bar h_{331}=\bar h_{231}=0.$$

\noindent From $\bar h_{221}=\bar h_{331}=\bar h_{231}=0$, the second equation in \eqref{3.1-17} and \eqref{3.1-22}, we have
$$\bar h_{1122}=0,$$
and then, by \eqref{3.1-14}, we know
\begin{equation*}
\begin{aligned}
0=&\bar \lambda_{1}\bar h_{1122}+\bar \lambda_{2}\bar h_{2222}+\bar \lambda_{3}\bar h_{3322}\\
 =&-\bar h^{2}_{112}-\bar h^{2}_{222}-\bar h^{2}_{332}-2\bar h^{2}_{122}-2\bar h^{2}_{132}-2\bar h^{2}_{232},\\
 =&-\bar h^{2}_{222}-\bar h^{2}_{332}-2\bar h^{2}_{232},
\end{aligned}
\end{equation*}
where $\bar \lambda_{2}=\bar \lambda_{3}=0$.
Thus,
$$\bar h_{222}=\bar h_{332}=\bar h_{232}=\bar h_{333}=0.$$
That is,
\begin{equation*}
\bar h_{ijk}=0,\ \ \text{for } \ i,j,k=1, 2, 3.
\end{equation*}

\noindent From \eqref{2.1-16} in Lemma \ref{lemma 2.2}, we have
\begin{equation*}
 0=S^{2}-\lambda \bar f_{3},
\end{equation*}
then, we obtain
$$\lambda^{2}=\bar H^{2}=\sup H^{2}, \ \ \lambda^{2}=S.$$

\noindent {\bf Subcase 1.2: $\bar \lambda_{1}=0$, $\bar \lambda_{2}=\bar \lambda_{3}\neq0$.}

\noindent Since $\bar H^2\neq 0$, we have that $\bar H^{2} = 2S$ and $\bar f_{3}= \bar\lambda_{2}S$.

\noindent From \eqref{3.1-11}, we have
\begin{equation*}
\bar H_{,k}=\bar \lambda_{k}\lim_{t\rightarrow\infty} \langle T, e_{k} \rangle(p_{t})=0 ,\ \ \text{for } \ k=1, 2, 3,
\end{equation*}
and then,
\begin{equation*}
\lim_{t\rightarrow\infty} \langle T, e_{2} \rangle(p_{t})=0, \ \ \lim_{t\rightarrow\infty} \langle T, e_{3} \rangle(p_{t})=0.
\end{equation*}

\noindent From \eqref{3.1-22}, we have
\begin{equation*}
\bar h_{111}=\bar h_{112}=\bar h_{113}=0, \ \ \bar h_{221}=-\bar h_{331}, \ \ \bar h_{222}=-\bar h_{332}, \ \ \bar h_{223}=-\bar h_{333}.
\end{equation*}

\noindent By \eqref{3.1-14} and \eqref{3.1-17}, we have that
\begin{equation*}
\begin{cases}
\begin{aligned}
&\bar \lambda_{2}(\bar h_{2211}+\bar h_{3311})=-2\bar h^{2}_{221}-2\bar h^{2}_{123},\\
&\bar \lambda^{3}_{2}(\bar h_{2211}+\bar h_{3311})=-6\lambda^{2}_{2}\bar h^{2}_{221}-6\lambda^{2}_{2}\bar h^{2}_{123},\\

&\bar \lambda_{2}(\bar h_{2222}+\bar h_{3322})=-2\bar h^{2}_{221}-2\bar h^{2}_{222}-2\bar h^{2}_{223}-2\bar h^{2}_{123}, \\
&\bar \lambda^{3}_{2}(\bar h_{2222}+\bar h_{3322})=-2\lambda^{2}_{2}(\bar h^{2}_{221}+3\bar h^{2}_{222}+3\bar h^{2}_{223}+\bar h^{2}_{123}),
\end{aligned}
\end{cases}
\end{equation*}
then,
\begin{equation*}
\bar h_{221}=0, \ \ \bar h_{123}=0, \ \ \bar h_{222}=0, \ \ \bar h_{223}=0.
\end{equation*}
Therefore,
\begin{equation*}
\bar h_{ijk}=0,\ \ \text{for } \ i,j,k=1, 2, 3.
\end{equation*}

\noindent From \eqref{2.1-16} in Lemma \ref{lemma 2.2}, we have
\begin{equation*}
 0=S^{2}-\lambda \bar f_{3},
\end{equation*}
then, we obtain
$$\lambda^{2}=\bar H^{2}=\sup H^{2}, \ \ \lambda^{2}=2S.$$

\noindent {\bf Case 2: $\bar \lambda_{1}\bar \lambda_{2}\bar \lambda_{3}\neq0$.}

\noindent According to the hypothesis,
we have that
\begin{equation*}
\bar H=\bar \lambda_{1}+2\bar \lambda_{2}\neq 0, \ \ \bar \lambda_{1}\neq\bar \lambda_{2}=\bar \lambda_{3}, \ \ \bar \lambda_{k} \neq 0,\ \ \text{for } \ k=1, 2, 3.
\end{equation*}

\noindent From \eqref{3.1-11} and $\bar \lambda_{k} \neq 0$ for $k=1, 2, 3$, we get

\begin{equation}\label{3.1-23}
\bar H_{,k}=\bar \lambda_{k}\lim_{t\rightarrow\infty} \langle T, e_{k} \rangle(p_{t})=0,\ \
\lim_{t\rightarrow\infty} \langle T, e_{k} \rangle(p_{t})=0,\ \ \text{for } \ k=1, 2, 3,
\end{equation}

\noindent From \eqref{3.1-12}, \eqref{3.1-14}, \eqref{3.1-17}, \eqref{3.1-22} and \eqref{3.1-23}, we know that

\begin{equation}\label{3.1-24}
\begin{cases}
\begin{aligned}
&\bar h_{1111}+\bar h_{2211}+\bar h_{3311}=(\bar H-\lambda)\bar \lambda^{2}_{1},\\
&\bar h_{1122}+\bar h_{2222}+\bar h_{3322}=(\bar H-\lambda)\bar \lambda^{2}_{2},\\
&\bar h_{1133}+\bar h_{2233}+\bar h_{3333}=(\bar H-\lambda)\bar \lambda^{2}_{2},\\
&\bar h_{1112}+\bar h_{2212}+\bar h_{3312}=0,\\
&\bar h_{1113}+\bar h_{2213}+\bar h_{3313}=0,
\end{aligned}
\end{cases}
\end{equation}

\begin{equation}\label{3.1-25}
\begin{cases}
\begin{aligned}
&\bar \lambda_{1}\bar h_{1111}+\bar \lambda_{2}(\bar h_{2211}+\bar h_{3311})=-2(\bar h^{2}_{221}+\bar h^{2}_{123}),\\
&\bar \lambda_{1}\bar h_{1122}+\bar \lambda_{2}(\bar h_{2222}+\bar h_{3322})=-2(\bar h^{2}_{222}+\bar h^{2}_{223})-2(\bar h^{2}_{221}+\bar h^{2}_{123}),\\
&\bar \lambda_{1}\bar h_{1133}+\bar \lambda_{2}(\bar h_{2233}+\bar h_{3333})=-2(\bar h^{2}_{222}+\bar h^{2}_{223})-2(\bar h^{2}_{221}+\bar h^{2}_{123}),\\
&\bar \lambda_{1}\bar h_{1112}+\bar \lambda_{2}(\bar h_{2212}+\bar h_{3312})=-2(\bar h_{221}\bar h_{222}+\bar h_{223}\bar h_{123}),\\
&\bar \lambda_{1}\bar h_{1113}+\bar \lambda_{2}(\bar h_{2213}+\bar h_{3313})=-2(\bar h_{221}\bar h_{223}-\bar h_{222}\bar h_{123}),
\end{aligned}
\end{cases}
\end{equation}
and
\begin{equation}\label{3.1-26}
\begin{cases}
\begin{aligned}
\bar \lambda^{3}_{1}\bar h_{1111}+\bar \lambda^{3}_{2}(\bar h_{2211}+\bar h_{3311})=&-6\bar\lambda^{2}_{2}(\bar h^{2}_{221}+\bar h^{2}_{123}),\\

\bar \lambda^{3}_{1}\bar h_{1122}+\bar \lambda^{3}_{2}(\bar h_{2222}+\bar h_{3322})=&-6\bar\lambda^{2}_{2}(\bar h^{2}_{222}+\bar h^{2}_{223})-2(\bar\lambda^{2}_{1}+\bar\lambda^{2}_{2}\\
&+\bar\lambda_{1}\bar\lambda_{2})(\bar h^{2}_{221}+\bar h^{2}_{123}),\\

\bar \lambda^{3}_{1}\bar h_{1133}+\bar \lambda^{3}_{2}(\bar h_{2233}+\bar h_{3333})=&-6\bar\lambda^{2}_{2}(\bar h^{2}_{222}+\bar h^{2}_{223})-2(\bar\lambda^{2}_{1}+\bar\lambda^{2}_{2}\\
&+\bar\lambda_{1}\bar\lambda_{2})(\bar h^{2}_{221}+\bar h^{2}_{123}),\\

\bar \lambda^{3}_{1}\bar h_{1112}+\bar \lambda^{3}_{2}(\bar h_{2212}+\bar h_{3312})=&-6\bar\lambda^{2}_{2}(\bar h_{221}\bar h_{222}+\bar h_{223}\bar h_{123}),\\

\bar \lambda^{3}_{1}\bar h_{1113}+\bar \lambda^{3}_{2}(\bar h_{2213}+\bar h_{3313})=&-6\bar\lambda^{2}_{2}(\bar h_{221}\bar h_{223}-\bar h_{222}\bar h_{123}).
\end{aligned}
\end{cases}
\end{equation}

\noindent From $\bar \lambda_{1} \neq \bar \lambda_{2}$, \eqref{3.1-24}, \eqref{3.1-25} and \eqref{3.1-26}, we get
\begin{equation*}
\bar h_{2212}+\bar h_{3312}= -\bar h_{1112},\ \
\bar h_{2213}+\bar h_{3313}= -\bar h_{1113},
\end{equation*}
and then,
\begin{equation}\label{3.1-27}
\bar h_{221}\bar h_{222}+\bar h_{223}\bar h_{123}=0, \ \
\bar h_{221}\bar h_{223}-\bar h_{222}\bar h_{123}=0, \ \
\bar h_{1112} = 0, \ \ \bar h_{1113} = 0.
\end{equation}

\noindent Besides, by \eqref{3.1-25} and \eqref{3.1-26}, we get

\begin{equation}\label{3.1-28}
\begin{cases}
\begin{aligned}
&\bar \lambda_{1}(\bar \lambda^{2}_{2}-\bar \lambda^{2}_{1})\bar h_{1111}= 4\bar\lambda^{2}_{2}(\bar h^{2}_{221}+\bar h^{2}_{123}),\\

&\bar \lambda_{1}(\bar \lambda^{2}_{2}-\bar \lambda^{2}_{1})\bar h_{1122}= 4\bar\lambda^{2}_{2}(\bar h^{2}_{222}+\bar h^{2}_{223})+2(\bar\lambda^{2}_{1}+\bar \lambda_{1}\bar \lambda_{2})(\bar h^{2}_{221}+\bar h^{2}_{123}),\\

&\bar \lambda_{1}(\bar \lambda^{2}_{2}-\bar \lambda^{2}_{1})\bar h_{1133}= 4\bar\lambda^{2}_{2}(\bar h^{2}_{222}+\bar h^{2}_{223})+2(\bar\lambda^{2}_{1}+\bar \lambda_{1}\bar \lambda_{2})(\bar h^{2}_{221}+\bar h^{2}_{123}).
\end{aligned}
\end{cases}
\end{equation}

\noindent Now we consider four subcases.

\noindent {\bf Subcase 2.1: $\bar h^{2}_{221}+\bar h^{2}_{123} \neq 0, \ \ \bar h^{2}_{222}+\bar h^{2}_{223} \neq 0$.}

\noindent From \eqref{3.1-27}, it is a contradiction.

\noindent {\bf Subcase 2.2: $\bar h^{2}_{221}+\bar h^{2}_{123} = 0, \ \ \bar h^{2}_{222}+\bar h^{2}_{223} = 0$.}

\noindent
From \eqref{3.1-22}, we know
\begin{equation}\label{3.1-29}
\bar h_{ijk}=0, \ \ \text{for } \ i,j,k=1, 2, 3,
\end{equation}
and then, by \eqref{2.1-16} in Lemma \ref{lemma 2.2}, we have
\begin{equation}\label{3.1-30}
0=S^{2}-\lambda \bar f_{3}.
\end{equation}

\noindent If $\bar \lambda_{1}+\bar \lambda_{2} = 0$, we have
\begin{equation}\label{3.1-31}
\bar H = -\bar \lambda_{1}, \ \ S = 3\bar \lambda^{2}_{1}, \ \ \bar f_{3}= -\bar \lambda^{3}_{1}.
\end{equation}

\noindent
From \eqref{3.1-30} and \eqref{3.1-31}, we know
\begin{equation}\label{3.1-32}
\lambda = -9\bar \lambda_{1}=9\bar H.
\end{equation}

\noindent
From \eqref{2.1-17}, \eqref{2.1-18}, \eqref{3.1-29}, \eqref{3.1-31} and \eqref{3.1-32}, we know
\begin{equation*}
\begin{aligned}
&\frac{1}{2}\lim_{t\rightarrow\infty}\Delta_{-V}\sum_{i, j,k}(h_{ijk})^{2}(p_{t})=\sum_{i,j,k,l}(\bar h_{ijkl})^{2},\\
&\frac{1}{2}\lim_{t\rightarrow\infty}\Delta_{-V}\sum_{i, j,k}(h_{ijk})^{2}(p_{t})\\
=&\frac{3}{4}\lambda S(S-\bar H^{2})(\bar H-\lambda)+\frac{9}{2}\lambda S\bar h_{11}\bar h_{22}\bar h_{33}
 -\frac{3}{2}\lambda^{2}\sum_{k}(\bar h_{22}\bar h_{33}\bar h^{2}_{1k}+\bar h_{11}\bar h_{33}\bar h^{2}_{2k}+\bar h_{11}\bar h_{22}\bar h^{2}_{3k})\\
=&\frac{3}{4}\lambda S(S-\bar H^{2})(\bar H-\lambda)+\frac{9}{2}\lambda\bar \lambda_{1}\bar \lambda_{2}\bar \lambda_{3} S
  -\frac{3}{2}\lambda^{2}\bar \lambda_{1}\bar \lambda_{2}\bar \lambda_{3} \bar H \\
=&-324\bar \lambda^{6}_{1},
\end{aligned}
\end{equation*}
and then,
\begin{equation*}
\sum_{i,j,k,l}(\bar h_{ijkl})^{2}=-324\bar \lambda^{6}_{1}< 0.
\end{equation*}
It is a contradiction.

\noindent If $\bar \lambda_{1}+\bar \lambda_{2} \neq 0$,
from \eqref{3.1-25} and \eqref{3.1-26},
we know that
\begin{equation*}
\begin{cases}
\begin{aligned}
&\bar h_{1111}= 0, \ \ \bar h_{2211}+\bar h_{3311} = 0,\\
&\bar h_{1122}= 0, \ \ \bar h_{2222}+\bar h_{3322} = 0,\\
&\bar h_{1133}= 0, \ \ \bar h_{2233}+\bar h_{3333} = 0,
\end{aligned}
\end{cases}
\end{equation*}
and then, by \eqref{3.1-24}, we have
\begin{equation}\label{3.1-33}
\lambda=\bar H.
\end{equation}

\noindent
From \eqref{3.1-30} and \eqref{3.1-33}, we know
\begin{equation*}
\bar \lambda_{1}=\bar \lambda_{2},
\end{equation*}
where $\bar H=\bar \lambda_{1}+2\bar \lambda_{2}$, $S=\bar \lambda^{2}_{1}+2\bar \lambda^{2}_{2}$ and $\bar f_{3}=\bar \lambda^{3}_{1}+2\bar \lambda^{3}_{2}$.
It is a contradiction.

\noindent {\bf Subcase 2.3: $\bar h^{2}_{221}+\bar h^{2}_{123} = 0, \ \ \bar h^{2}_{222}+\bar h^{2}_{223} \neq 0$.}

\noindent
From \eqref{3.1-28}, we know
\begin{equation*}
\begin{cases}
\begin{aligned}
&\bar \lambda_{1}(\bar \lambda^{2}_{2}-\bar \lambda^{2}_{1})\bar h_{1111}= 0,\\

&\bar \lambda_{1}(\bar \lambda^{2}_{2}-\bar \lambda^{2}_{1})\bar h_{1122}= 4\bar\lambda^{2}_{2}(\bar h^{2}_{222}+\bar h^{2}_{223}),\\

&\bar \lambda_{1}(\bar \lambda^{2}_{2}-\bar \lambda^{2}_{1})\bar h_{1133}= 4\bar\lambda^{2}_{2}(\bar h^{2}_{222}+\bar h^{2}_{223}),
\end{aligned}
\end{cases}
\end{equation*}
and then,
\begin{equation}\label{3.1-34}
\bar \lambda_{1}+\bar \lambda_{2}\neq 0, \ \ \bar h_{1111}= 0, \ \ \bar h_{1122}=\frac{4\bar\lambda^{2}_{2}}{\bar \lambda_{1}(\bar \lambda^{2}_{2}-\bar \lambda^{2}_{1})}(\bar h^{2}_{222}+\bar h^{2}_{223}).
\end{equation}

\noindent
From $\bar h_{1111}= 0$ and the first equation in \eqref{3.1-25}, we know
\begin{equation*}
\bar h_{2211}+\bar h_{3311}=0,
\end{equation*}
and then, by \eqref{3.1-24}, we have
\begin{equation}\label{3.1-35}
\bar H_{,11}=0, \ \ \lambda=\bar H, \ \ \bar H_{,22}=0.
\end{equation}

\noindent
From \eqref{3.1-25} and \eqref{3.1-35}, we know
\begin{equation}\label{3.1-36}
\bar h_{1122}= \frac{2}{\bar\lambda_{2}-\bar\lambda_{1}}(\bar h^{2}_{222}+\bar h^{2}_{223}),
\end{equation}

\noindent From \eqref{3.1-34} and \eqref{3.1-36}, we have
$$\bar\lambda_{1}=\bar\lambda_{2}.$$
It is a contradiction.

\noindent {\bf Subcase 2.4: $\bar h^{2}_{221}+\bar h^{2}_{123} \neq 0, \ \ \bar h^{2}_{222}+\bar h^{2}_{223} = 0$.}

\noindent
From \eqref{3.1-28}, we know
\begin{equation*}
\begin{cases}
\begin{aligned}
&\bar \lambda_{1}(\bar \lambda^{2}_{2}-\bar \lambda^{2}_{1})\bar h_{1111}= 4\bar\lambda^{2}_{2}(\bar h^{2}_{221}+\bar h^{2}_{123}),\\

&\bar \lambda_{1}(\bar \lambda^{2}_{2}-\bar \lambda^{2}_{1})\bar h_{1122}= 2(\bar\lambda^{2}_{1}+\bar \lambda_{1}\bar \lambda_{2})(\bar h^{2}_{221}+\bar h^{2}_{123}),\\

&\bar \lambda_{1}(\bar \lambda^{2}_{2}-\bar \lambda^{2}_{1})\bar h_{1133}= 2(\bar\lambda^{2}_{1}+\bar \lambda_{1}\bar \lambda_{2})(\bar h^{2}_{221}+\bar h^{2}_{123}),
\end{aligned}
\end{cases}
\end{equation*}
and then,
\begin{equation}\label{3.1-37}
\begin{cases}
\begin{aligned}
&\bar \lambda_{1}+\bar \lambda_{2} \neq 0, \ \ \bar h_{1111}=\frac{4\bar\lambda^{2}_{2}}{\bar \lambda_{1}(\bar \lambda^{2}_{2}-\bar \lambda^{2}_{1})}(\bar h^{2}_{221}+\bar h^{2}_{123}), \\
&\bar h_{1122}=\bar h_{1133}=\frac{2}{\bar\lambda_{2}-\bar \lambda_{1}}(\bar h^{2}_{221}+\bar h^{2}_{123}).
\end{aligned}
\end{cases}
\end{equation}
\noindent
From \eqref{3.1-24}, \eqref{3.1-25} and \eqref{3.1-37}, we know
\begin{equation*}
\begin{aligned}
-2(\bar h^{2}_{221}+\bar h^{2}_{123})&=\bar \lambda_{1}\bar h_{1122}+\bar \lambda_{2}((\bar H-\lambda)\bar \lambda^{2}_{2}-\bar h_{1122}) \\
&=(\bar H-\lambda)\bar \lambda^{3}_{2}+(\bar \lambda_{1}-\bar \lambda_{2})\bar h_{1122} \\
&=(\bar H-\lambda)\bar \lambda^{3}_{2}+(\bar \lambda_{1}-\bar \lambda_{2}) \cdot \frac{2}{\bar \lambda_{2}-\bar \lambda_{1}}(\bar h^{2}_{221}+\bar h^{2}_{123}) \\
&=(\bar H-\lambda)\bar \lambda^{3}_{2}-2(\bar h^{2}_{221}+\bar h^{2}_{123}),
\end{aligned}
\end{equation*}
and then, by \eqref{3.1-24}, we have
\begin{equation}\label{3.1-38}
\lambda=\bar H, \ \ \bar H_{,11}=0.
\end{equation}

\noindent From \eqref{3.1-25} and \eqref{3.1-38}, we have

\begin{equation}\label{3.1-39}
\bar h_{1111}=\frac{2}{\bar\lambda_{2}-\bar\lambda_{1}}(\bar h^{2}_{221}+\bar h^{2}_{123}).
\end{equation}

\noindent
From \eqref{3.1-37} and \eqref{3.1-39}, we know
$$\bar\lambda_{1}=\bar\lambda_{2}.$$
It is a contradiction.

\vskip2mm
\noindent {\bf 3. The values of the principal curvature $\bar \lambda_1$, $\bar \lambda_{2}$ and $\bar \lambda_3$ are not equal to each other.}

\noindent {\bf Case 1: $\bar \lambda_{1}\bar \lambda_{2}\bar \lambda_{3}=0$.}

\noindent Without loss of generality, we assume that $\bar \lambda_{3}=0$, we know that $\bar \lambda_{1}\neq 0$, $\bar \lambda_{2}\neq 0$, $\bar \lambda_{1} -\bar \lambda_{2}\neq 0$ and $\bar H =\bar \lambda_{1}+\bar \lambda_{2}\neq 0$.

\noindent From \eqref{3.1-13} and \eqref{3.1-16}, we have that
\begin{equation}\label{3.1-40}
\bar h_{11k}=\bar h_{22k}=0, \ \ \text{for } \ k=1, 2, 3.
\end{equation}

\noindent From \eqref{3.1-10} and \eqref{3.1-40}, we have
\begin{equation}\label{3.1-41}
\bar h_{33k}=0, \ \ \text{for } \ k=1, 2, 3.
\end{equation}

\noindent By \eqref{3.1-14}, \eqref{3.1-17}, \eqref{3.1-40} and \eqref{3.1-41}, we have
\begin{equation*}
\begin{cases}
\begin{aligned}
&\bar \lambda_{1}\bar h_{1111}+\bar \lambda_{2}\bar h_{2211}=-2\bar h^{2}_{123}, \ \ \bar \lambda^{3}_{1}\bar h_{1111}+\bar \lambda^{3}_{2}\bar h_{2211}=-2\lambda^{2}_{2}\bar h^{2}_{123},\\

&\bar \lambda_{1}\bar h_{1122}+\bar \lambda_{2}\bar h_{2222}=-2\bar h^{2}_{123}, \ \ \bar \lambda^{3}_{1}\bar h_{1122}+\bar \lambda^{3}_{2}\bar h_{2222}=-2\lambda^{2}_{1}\bar h^{2}_{123},
\end{aligned}
\end{cases}
\end{equation*}
and then,
\begin{equation}\label{3.1-42}
\bar h_{1111}=0, \ \ \bar h_{2211}=-\frac{2\bar h^{2}_{123}}{\bar \lambda_{2}}, \ \ \bar h_{1122}=-\frac{2\bar h^{2}_{123}}{\bar \lambda_{1}},\ \ \bar h_{2222}=0.
\end{equation}

\noindent From \eqref{3.1-15} and \eqref{3.1-42}, we know
\begin{equation*}
\bar h_{1122}-\bar h_{2211}=-\bar \lambda_{1}\bar \lambda_{2}(\bar \lambda_{1}-\bar \lambda_{2})=\frac{2(\bar \lambda_{1}-\bar \lambda_{2})\bar h^{2}_{123}}{\bar \lambda_{1}\bar \lambda_{2}},
\end{equation*}
and then,
\begin{equation*}
\bar h^{2}_{123}=\frac{-\bar \lambda^{2}_{1}\bar \lambda^{2}_{2}}{2}, \ \ \bar h^{2}_{123}=0,\ \ \ \bar \lambda_{1}\bar \lambda_{2}=0.
\end{equation*}
It is a contradiction.

\noindent {\bf Case 2: $\bar \lambda_{1}\bar \lambda_{2}\bar \lambda_{3}\neq 0$.}

\noindent From \eqref{3.1-11}, we have that
\begin{equation}\label{3.1-43}
\lim_{m\rightarrow\infty} \langle T,e_{k} \rangle(p_{t})=0,\ \ \text{for } \ k=1, 2, 3.
\end{equation}

\noindent From \eqref{3.1-10}, \eqref{3.1-13} and \eqref{3.1-16}, we have that
\begin{equation}\label{3.1-44}
\bar h_{11k}=\bar h_{22k}=\bar h_{33k}=0,\ \ \text{for } \ k=1, 2, 3.
\end{equation}

\noindent From \eqref{3.1-12}, \eqref{3.1-14}, \eqref{3.1-17}, \eqref{3.1-43} and \eqref{3.1-44}, we have that

\begin{equation}\label{3.1-45}
\begin{cases}
\begin{aligned}
&\bar h_{1111}+\bar h_{2211}+\bar h_{3311}=(\bar H-\lambda)\bar \lambda^{2}_{1},\\
&\bar h_{1122}+\bar h_{2222}+\bar h_{3322}=(\bar H-\lambda)\bar \lambda^{2}_{2},\\
&\bar h_{1133}+\bar h_{2233}+\bar h_{3333}=(\bar H-\lambda)\bar \lambda^{2}_{3},\\
&\bar h_{1112}+\bar h_{2212}+\bar h_{3312}=0,\\
&\bar h_{1113}+\bar h_{2213}+\bar h_{3313}=0,\\
&\bar h_{1123}+\bar h_{2223}+\bar h_{3323}=0,
\end{aligned}
\end{cases}
\end{equation}

\begin{equation}\label{3.1-46}
\begin{cases}
\begin{aligned}
&\bar \lambda_{1}\bar h_{1111}+\bar \lambda_{2}\bar h_{2211}+\bar \lambda_{3}\bar h_{3311}=-2\bar h^{2}_{231},\\
&\bar \lambda_{1}\bar h_{1122}+\bar \lambda_{2}\bar h_{2222}+\bar \lambda_{3}\bar h_{3322}=-2\bar h^{2}_{132},\\
&\bar \lambda_{1}\bar h_{1133}+\bar \lambda_{2}\bar h_{2233}+\bar \lambda_{3}\bar h_{3333}=-2\bar h^{2}_{123},\\
&\bar \lambda_{1}\bar h_{1112}+\bar \lambda_{2}\bar h_{2212}+\bar \lambda_{3}\bar h_{3312}=0,\\
&\bar \lambda_{1}\bar h_{1113}+\bar \lambda_{2}\bar h_{2213}+\bar \lambda_{3}\bar h_{3313}=0,\\
&\bar \lambda_{1}\bar h_{1123}+\bar \lambda_{2}\bar h_{2223}+\bar \lambda_{3}\bar h_{3323}=0,
\end{aligned}
\end{cases}
\end{equation}
and
\begin{equation}\label{3.1-47}
\begin{cases}
\begin{aligned}
&\bar \lambda^{3}_{1}\bar h_{1111}+\bar \lambda^{3}_{2}\bar h_{2211}+\bar \lambda^{3}_{3}\bar h_{3311}
=-2(\bar\lambda^{2}_{2}+\bar\lambda^{2}_{3}+\bar\lambda_{2}\bar\lambda_{3})\bar h^{2}_{123},\\

&\bar \lambda^{3}_{1}\bar h_{1122}+\bar \lambda^{3}_{2}\bar h_{2222}+\bar \lambda^{3}_{3}\bar h_{3322}
=-2(\bar\lambda^{2}_{1}+\bar\lambda^{2}_{3}+\bar\lambda_{1}\bar\lambda_{3})\bar h^{2}_{123},\\

&\bar \lambda^{3}_{1}\bar h_{1133}+\bar \lambda^{3}_{2}\bar h_{2233}+\bar \lambda^{3}_{3}\bar h_{3333}
=-2(\bar\lambda^{2}_{1}+\bar\lambda^{2}_{2}+\bar\lambda_{1}\bar\lambda_{2})\bar h^{2}_{123},\\

&\bar \lambda^{3}_{1}\bar h_{1112}+\bar \lambda^{3}_{2}\bar h_{2212}+\bar \lambda^{3}_{3}\bar h_{3312}=0,\\

&\bar \lambda^{3}_{1}\bar h_{1113}+\bar \lambda^{3}_{2}\bar h_{2213}+\bar \lambda^{3}_{3}\bar h_{3313}=0,\\

&\bar \lambda^{3}_{1}\bar h_{1123}+\bar \lambda^{3}_{2}\bar h_{2223}+\bar \lambda^{3}_{3}\bar h_{3323}=0.
\end{aligned}
\end{cases}
\end{equation}

Therefore,
\begin{equation}\label{3.1-48}
\begin{cases}
\begin{aligned}
&\bar h_{1111}=\frac{\bar \lambda_{2}\bar \lambda_{3}(\bar \lambda_{2}+\bar \lambda_{3})}{\bar H(\bar \lambda_{1}-\bar \lambda_{2})(\bar \lambda_{1}-\bar \lambda_{3})}\cdot(\bar H-\lambda)\bar \lambda^{2}_{1}, \\

&\bar h_{2211}=\frac{-2\bar h^{2}_{123}}{\bar \lambda_{2}-\bar \lambda_{3}}+\frac{\bar \lambda_{1}\bar \lambda_{3}(\bar \lambda_{1}+\bar \lambda_{3})}{\bar H(\bar \lambda_{2}-\bar \lambda_{1})(\bar \lambda_{2}-\bar \lambda_{3})}\cdot(\bar H-\lambda)\bar \lambda^{2}_{1}, \\

&\bar h_{3311}=\frac{-2\bar h^{2}_{123}}{\bar \lambda_{3}-\bar \lambda_{2}}+\frac{\bar \lambda_{1}\bar \lambda_{2}(\bar \lambda_{1}+\bar \lambda_{2})}{\bar H(\bar \lambda_{3}-\bar \lambda_{1})(\bar \lambda_{3}-\bar \lambda_{2})}\cdot(\bar H-\lambda)\bar \lambda^{2}_{1}, \\

&\bar h_{2222}=\frac{\bar \lambda_{1}\bar \lambda_{3}(\bar \lambda_{1}+\bar \lambda_{3})}{\bar H(\bar \lambda_{2}-\bar \lambda_{1})(\bar \lambda_{2}-\bar \lambda_{3})}\cdot(\bar H-\lambda)\bar \lambda^{2}_{2}, \\

&\bar h_{1122}=\frac{-2\bar h^{2}_{123}}{\bar \lambda_{1}-\bar \lambda_{3}}+\frac{\bar \lambda_{2}\bar \lambda_{3}(\bar \lambda_{2}+\bar \lambda_{3})}{\bar H(\bar \lambda_{1}-\bar \lambda_{2})(\bar \lambda_{1}-\bar \lambda_{3})}\cdot(\bar H-\lambda)\bar \lambda^{2}_{2}, \\

&\bar h_{3322}=\frac{-2\bar h^{2}_{123}}{\bar \lambda_{3}-\bar \lambda_{1}}+\frac{\bar \lambda_{1}\bar \lambda_{2}(\bar \lambda_{1}+\bar \lambda_{2})}{\bar H(\bar \lambda_{3}-\bar \lambda_{1})(\bar \lambda_{3}-\bar \lambda_{2})}\cdot(\bar H-\lambda)\bar \lambda^{2}_{2}, \\

&\bar h_{3333}=\frac{\bar \lambda_{1}\bar \lambda_{2}(\bar \lambda_{1}+\bar \lambda_{2})}{\bar H(\bar \lambda_{3}-\bar \lambda_{1})(\bar \lambda_{3}-\bar \lambda_{2})}\cdot(\bar H-\lambda)\bar \lambda^{2}_{3}, \\

&\bar h_{1133}=\frac{-2\bar h^{2}_{123}}{\bar \lambda_{1}-\bar \lambda_{2}}+ \frac{\bar \lambda_{2}\bar \lambda_{3}(\bar \lambda_{2}+\bar \lambda_{3})}{\bar H(\bar \lambda_{1}-\bar \lambda_{2})(\bar \lambda_{1}-\bar \lambda_{3})}\cdot(\bar H-\lambda)\bar \lambda^{2}_{3}, \\

&\bar h_{2233}=\frac{-2\bar h^{2}_{123}}{\bar \lambda_{2}-\bar \lambda_{1}}+\frac{\bar \lambda_{1}\bar \lambda_{3}(\bar \lambda_{1}+\bar \lambda_{3})}{\bar H(\bar \lambda_{2}-\bar \lambda_{1})(\bar \lambda_{2}-\bar \lambda_{3})}\cdot(\bar H-\lambda)\bar \lambda^{2}_{3},\\

&\bar h_{1112}=\bar h_{2212}=\bar h_{3312}=0, \ \ \bar h_{1113}=\bar h_{2213}=\bar h_{3313}=0, \\
&\bar h_{1123}=\bar h_{2223}=\bar h_{3323}=0.
\end{aligned}
\end{cases}
\end{equation}

\noindent From \eqref{3.1-15} and \eqref{3.1-48}, we have that
\begin{equation}\label{3.1-49}
\begin{cases}
\begin{aligned}
&\frac{2\bar h^{2}_{123}(\bar \lambda_{1}-\bar \lambda_{2})}{(\bar \lambda_{1}-\bar \lambda_{3})(\bar \lambda_{2}-\bar \lambda_{3})}
+ \frac{(\bar H-\lambda)\bar \lambda_{3}\bigg(\bar \lambda^{3}_{1}(\bar \lambda^{2}_{1}-\bar \lambda^{2}_{3})+\bar \lambda^{3}_{2}(\bar \lambda^{2}_{2}-\bar \lambda^{2}_{3})\bigg)}{\bar H(\bar \lambda_{1}-\bar \lambda_{2})(\bar \lambda_{1}-\bar \lambda_{3})(\bar \lambda_{2}-\bar \lambda_{3})}\\
&=-\bar \lambda_{1}\bar \lambda_{2}(\bar \lambda_{1}-\bar \lambda_{2}),\\

&\frac{-2\bar h^{2}_{123}(\bar \lambda_{1}-\bar \lambda_{3})}{(\bar \lambda_{1}-\bar \lambda_{2})(\bar \lambda_{2}-\bar \lambda_{3})}
+ \frac{(\bar H-\lambda)\bar \lambda_{2}\bigg(\bar \lambda^{3}_{3}(\bar \lambda^{2}_{2}-\bar \lambda^{2}_{3})-\bar \lambda^{3}_{1}(\bar \lambda^{2}_{1}-\bar \lambda^{2}_{2})\bigg)}{\bar H(\bar \lambda_{1}-\bar \lambda_{2})(\bar \lambda_{1}-\bar \lambda_{3})(\bar \lambda_{2}-\bar \lambda_{3})}\\
&=-\bar \lambda_{1}\bar \lambda_{3}(\bar \lambda_{1}-\bar \lambda_{3}),\\

&\frac{2\bar h^{2}_{123}(\bar \lambda_{2}-\bar \lambda_{3})}{(\bar \lambda_{1}-\bar \lambda_{2})(\bar \lambda_{1}-\bar \lambda_{3})}
- \frac{(\bar H-\lambda)\bar \lambda_{1}\bigg(\bar \lambda^{3}_{3}(\bar \lambda^{2}_{1}-\bar \lambda^{2}_{3})+\bar \lambda^{3}_{2}(\bar \lambda^{2}_{1}-\bar \lambda^{2}_{2})\bigg)}{\bar H(\bar \lambda_{1}-\bar \lambda_{2})(\bar \lambda_{1}-\bar \lambda_{3})(\bar \lambda_{2}-\bar \lambda_{3})}\\
&=-\bar \lambda_{2}\bar \lambda_{3}(\bar \lambda_{2}-\bar \lambda_{3}),
\end{aligned}
\end{cases}
\end{equation}
and then,
\begin{equation*}
\begin{cases}
\begin{aligned}
&2\bar h^{2}_{123}\cdot\bigg(\bar \lambda_{3}(\bar \lambda_{1}-\bar \lambda_{2})+\bar \lambda_{2}(\bar \lambda_{1}-\bar \lambda_{3})\bigg)
+\frac{(\bar H-\lambda)}{\bar H}\cdot\bigg(\frac{\bar \lambda^{3}_{1}\bar \lambda^{2}_{3}(\bar \lambda^{2}_{1}-\bar \lambda^{2}_{3})+\bar \lambda^{3}_{2}\bar \lambda^{2}_{3}(\bar \lambda^{2}_{2}-\bar \lambda^{2}_{3})}{\bar \lambda_{1}-\bar \lambda_{2}}\\
&+\frac{\bar \lambda^{2}_{2}\bar \lambda^{3}_{3}(\bar \lambda^{2}_{3}-\bar \lambda^{2}_{2})+\bar \lambda^{3}_{1}\bar \lambda^{2}_{2}(\bar \lambda^{2}_{1}-\bar \lambda^{2}_{2})}{\bar \lambda_{1}-\bar \lambda_{3}}\bigg)=0, \\

&2\bar h^{2}_{123}\cdot\bigg(\bar \lambda_{1}(\bar \lambda_{3}-\bar \lambda_{2})+\bar \lambda_{2}(\bar \lambda_{3}-\bar \lambda_{1})\bigg)
+\frac{(\bar H-\lambda)}{\bar H}\cdot\bigg(\frac{\bar \lambda^{2}_{1}\bar \lambda^{3}_{3}(\bar \lambda^{2}_{3}-\bar \lambda^{2}_{1})+\bar \lambda^{2}_{1}\bar \lambda^{3}_{2}(\bar \lambda^{2}_{2}-\bar \lambda^{2}_{1})}{\bar \lambda_{3}-\bar \lambda_{2}}\\
&+\frac{\bar \lambda^{2}_{2}\bar \lambda^{3}_{3}(\bar \lambda^{2}_{3}-\bar \lambda^{2}_{2})+\bar \lambda^{3}_{1}\bar \lambda^{2}_{2}(\bar \lambda^{2}_{1}-\bar \lambda^{2}_{2})}{\bar \lambda_{3}-\bar \lambda_{1}}\bigg)=0, \\

&2\bar h^{2}_{123}\cdot\bigg(\bar \lambda_{3}(\bar \lambda_{2}-\bar \lambda_{1})+\bar \lambda_{1}(\bar \lambda_{2}-\bar \lambda_{3})\bigg)
+\frac{(\bar H-\lambda)}{\bar H}\cdot\bigg(\frac{\bar \lambda^{3}_{1}\bar \lambda^{2}_{3}(\bar \lambda^{2}_{1}-\bar \lambda^{2}_{3})+\bar \lambda^{3}_{2}\bar \lambda^{2}_{3}(\bar \lambda^{2}_{2}-\bar \lambda^{2}_{3})}{\bar \lambda_{2}-\bar \lambda_{1}}\\
&+\frac{\bar \lambda^{2}_{1}\bar \lambda^{3}_{3}(\bar \lambda^{2}_{3}-\bar \lambda^{2}_{1})+\bar \lambda^{2}_{1}\bar \lambda^{3}_{2}(\bar \lambda^{2}_{2}-\bar \lambda^{2}_{1})}{\bar \lambda_{2}-\bar \lambda_{3}}\bigg)=0.
\end{aligned}
\end{cases}
\end{equation*}
That is,
$$AX=0,$$
where
\begin{equation*}
A=\left(
\begin{array}{cc}
\bar \lambda_{3}(\bar \lambda_{1}-\bar \lambda_{2})+\bar \lambda_{2}(\bar \lambda_{1}-\bar \lambda_{3})
&\frac{\bar \lambda^{3}_{1}\bar \lambda^{2}_{3}(\bar \lambda^{2}_{1}-\bar \lambda^{2}_{3})+\bar \lambda^{3}_{2}\bar \lambda^{2}_{3}(\bar \lambda^{2}_{2}-\bar \lambda^{2}_{3})}{\bar \lambda_{1}-\bar \lambda_{2}}
+\frac{\bar \lambda^{2}_{2}\bar \lambda^{3}_{3}(\bar \lambda^{2}_{3}-\bar \lambda^{2}_{2})+\bar \lambda^{3}_{1}\bar \lambda^{2}_{2}(\bar \lambda^{2}_{1}-\bar \lambda^{2}_{2})}{\bar \lambda_{1}-\bar \lambda_{3}}  \\
\bar \lambda_{1}(\bar \lambda_{3}-\bar \lambda_{2})+\bar \lambda_{2}(\bar \lambda_{3}-\bar \lambda_{1}) & \frac{\bar \lambda^{2}_{1}\bar \lambda^{3}_{3}(\bar \lambda^{2}_{3}-\bar \lambda^{2}_{1})+\bar \lambda^{2}_{1}\bar \lambda^{3}_{2}(\bar \lambda^{2}_{2}-\bar \lambda^{2}_{1})}{\bar \lambda_{3}-\bar \lambda_{2}}
+\frac{\bar \lambda^{2}_{2}\bar \lambda^{3}_{3}(\bar \lambda^{2}_{3}-\bar \lambda^{2}_{2})+\bar \lambda^{3}_{1}\bar \lambda^{2}_{2}(\bar \lambda^{2}_{1}-\bar \lambda^{2}_{2})}{\bar \lambda_{3}-\bar \lambda_{1}}
\end{array}
\right),
\end{equation*}
and
\begin{equation*}
X=\left(
\begin{array}{cc}
2\bar h^{2}_{123}\\
\frac{\bar H-\lambda}{\bar H}
\end{array}
\right),
\end{equation*}

\noindent By a direct calculation, we have
\begin{equation*}
\begin{aligned}
det(A)
=&\frac{1}{(\bar \lambda_{1}-\bar \lambda_{2})(\bar \lambda_{1}-\bar \lambda_{3})(\bar \lambda_{2}-\bar \lambda_{3})}\cdot(2\bar \lambda^{7}_{1}\bar \lambda^{4}_{2}-2\bar \lambda^{6}_{1}\bar \lambda^{5}_{2}-2\bar \lambda^{5}_{1}\bar \lambda^{6}_{2}+2\bar \lambda^{4}_{1}\bar \lambda^{7}_{2}-2\bar \lambda^{7}_{1}\bar \lambda^{3}_{2}\bar \lambda_{3}\\
&+4\bar \lambda^{5}_{1}\bar \lambda^{5}_{2}\bar \lambda_{3}-2\bar \lambda^{3}_{1}\bar \lambda^{7}_{2}\bar \lambda_{3}+4\bar \lambda^{7}_{1}\bar \lambda^{2}_{2}\bar \lambda^{2}_{3}-4\bar \lambda^{6}_{1}\bar \lambda^{3}_{2}\bar \lambda^{2}_{3}+2\bar \lambda^{5}_{1}\bar \lambda^{4}_{2}\bar \lambda^{2}_{3}+2\bar \lambda^{4}_{1}\bar \lambda^{5}_{2}\bar \lambda^{2}_{3}\\
&-4\bar \lambda^{3}_{1}\bar \lambda^{6}_{2}\bar \lambda^{2}_{3}+4\bar \lambda^{2}_{1}\bar \lambda^{7}_{2}\bar \lambda^{2}_{3}-2\bar \lambda^{7}_{1}\bar \lambda_{2}\bar \lambda^{3}_{3}-4\bar \lambda^{6}_{1}\bar \lambda^{2}_{2}\bar \lambda^{3}_{3}-4\bar \lambda^{2}_{1}\bar \lambda^{6}_{2}\bar \lambda^{3}_{3}-2\bar \lambda_{1}\bar \lambda^{7}_{2}\bar \lambda^{3}_{3}\\
&+2\bar \lambda^{7}_{1}\bar \lambda^{4}_{3}+2\bar \lambda^{5}_{1}\bar \lambda^{2}_{2}\bar \lambda^{4}_{3}+2\bar \lambda^{2}_{1}\bar \lambda^{5}_{2}\bar \lambda^{4}_{3}+2\bar \lambda^{7}_{2}\bar \lambda^{4}_{3}-2\bar \lambda^{6}_{1}\bar \lambda^{5}_{3}+4\bar \lambda^{5}_{1}\bar \lambda_{2}\bar \lambda^{5}_{3}+2\bar \lambda^{4}_{1}\bar \lambda^{2}_{2}\bar \lambda^{5}_{3}\\
&+2\bar \lambda^{2}_{1}\bar \lambda^{4}_{2}\bar \lambda^{5}_{3}+4\bar \lambda_{1}\bar \lambda^{5}_{2}\bar \lambda^{5}_{3}-2\bar \lambda^{6}_{2}\bar \lambda^{5}_{3}-2\bar \lambda^{5}_{1}\bar \lambda^{6}_{3}-4\bar \lambda^{3}_{1}\bar \lambda^{2}_{2}\bar \lambda^{6}_{3}-4\bar \lambda^{2}_{1}\bar \lambda^{3}_{2}\bar \lambda^{6}_{3}\\
&-2\bar \lambda^{5}_{2}\bar \lambda^{6}_{3}+2\bar \lambda^{4}_{1}\bar \lambda^{7}_{3}-2\bar \lambda^{3}_{1}\bar \lambda_{2}\bar \lambda^{7}_{3}+4\bar \lambda^{2}_{1}\bar \lambda^{2}_{2}\bar \lambda^{7}_{3}-2\bar \lambda_{1}\bar \lambda^{3}_{2}\bar \lambda^{7}_{3}+2\bar \lambda^{4}_{2}\bar \lambda^{7}_{3})\\
=&\frac{2}{(\bar \lambda_{1}-\bar \lambda_{2})(\bar \lambda_{1}-\bar \lambda_{3})(\bar \lambda_{2}-\bar \lambda_{3})}\cdot(\bar \lambda^{2}_{1}\bar \lambda^{2}_{2}-\bar \lambda^{2}_{1}\bar \lambda_{2}\bar \lambda_{3}-\bar \lambda_{1}\bar \lambda^{2}_{2}\bar \lambda_{3}+\bar \lambda^{2}_{1}\bar \lambda^{2}_{3}-\bar \lambda_{1}\bar \lambda_{2}\bar \lambda^{2}_{3}\\
&+\bar \lambda^{2}_{2}\bar \lambda^{2}_{3})\cdot
(\bar \lambda^{5}_{1}\bar \lambda^{2}_{2}-\bar \lambda^{4}_{1}\bar \lambda^{3}_{2}-\bar \lambda^{3}_{1}\bar \lambda^{4}_{2}+\bar \lambda^{2}_{1}\bar \lambda^{5}_{2}+\bar \lambda^{5}_{1}\bar \lambda^{2}_{3}+\bar \lambda^{5}_{2}\bar \lambda^{2}_{3}-\bar \lambda^{4}_{1}\bar \lambda^{3}_{3}-\bar \lambda^{4}_{2}\bar \lambda^{3}_{3}\\
&-\bar \lambda^{3}_{1}\bar \lambda^{4}_{3}-\bar \lambda^{3}_{2}\bar \lambda^{4}_{3}+\bar \lambda^{2}_{1}\bar \lambda^{5}_{3}+\bar \lambda^{2}_{2}\bar \lambda^{5}_{3})\\
=&\frac{2}{(\bar \lambda_{1}-\bar \lambda_{2})(\bar \lambda_{1}-\bar \lambda_{3})(\bar \lambda_{2}-\bar \lambda_{3})}\cdot\bigg((\bar \lambda^{2}_{1} +\bar \lambda^{2}_{2}+\bar \lambda^{2}_{3})(\bar \lambda^{5}_{1} +\bar \lambda^{5}_{2}+\bar \lambda^{5}_{3})-(\bar \lambda^{3}_{1} +\bar \lambda^{3}_{2}\\
&+\bar \lambda^{3}_{3})(\bar \lambda^{4}_{1} +\bar \lambda^{4}_{2}+\bar \lambda^{4}_{3})\bigg)\cdot\bigg(\bar \lambda^{2}_{1}(\bar \lambda_{2}-\bar \lambda_{3})^{2}+\bar \lambda^{2}_{2}(\bar \lambda_{1}-\bar \lambda_{3})^{2}+\bar \lambda^{2}_{3}(\bar \lambda_{2}-\bar \lambda_{1})^{2}\bigg).
\end{aligned}
\end{equation*}

\noindent When
$$(\bar \lambda^{2}_{1} +\bar \lambda^{2}_{2}+\bar \lambda^{2}_{3})(\bar \lambda^{5}_{1} +\bar \lambda^{5}_{2}+\bar \lambda^{5}_{3})
-(\bar \lambda^{3}_{1} +\bar \lambda^{3}_{2}+\bar \lambda^{3}_{3})(\bar \lambda^{4}_{1} +\bar \lambda^{4}_{2}+\bar \lambda^{4}_{3})\neq 0,$$
that is  $$S\bar f_{5}-\bar f_{3}f_{4}\neq 0,$$
we have that the matrix $A$ is nondegenerate, and then
 $$\bar h^{2}_{123}=0, \ \ \lambda=\bar H.$$
That is, $$\bar h_{ijk}=0, \ \ i,j,k=1, 2, 3.$$

\noindent From \eqref{2.1-16} and \eqref{2.1-17} in lemma \ref{lemma 2.2}, we obtain
\begin{equation*}
\begin{aligned}
&\sum_{i,j,k}h_{ijk}^{2}+S^{2}-\lambda f_{3}=0, \\
&2\sum_{i,j,k,l,m}h_{ijm}h_{jkm}h_{kl}h_{li}+\sum_{i,j,k,l,m}h_{ijm}h_{jk}h_{klm}h_{li}+Sf_{4}-\lambda f_{5}=0.
\end{aligned}
\end{equation*}
Specifically,
\begin{equation*}
\begin{aligned}
&\sum_{i,j,k}\bar h_{ijk}^{2}+S^{2}-\lambda \bar f_{3}=0, \ \ S^{2}-\lambda \bar f_{3}=0, \\
&2\sum_{i,j,k,l,m}\bar h_{ijm}\bar h_{jkm}\bar h_{kl}\bar h_{li}+\sum_{i,j,k,l,m}\bar h_{ijm}\bar h_{jk}\bar h_{klm}\bar h_{li}+Sf_{4}-\lambda \bar f_{5}=0, \ \ Sf_{4}-\lambda \bar f_{5}=0,
\end{aligned}
\end{equation*}
and then, $S\bar f_{5}-\bar f_{3}f_{4}=0$. This contradicts the hypothesis.

\noindent  When
$$(\bar \lambda^{2}_{1} +\bar \lambda^{2}_{2}+\bar \lambda^{2}_{3})(\bar \lambda^{5}_{1} +\bar \lambda^{5}_{2}+\bar \lambda^{5}_{3})
-(\bar \lambda^{3}_{1} +\bar \lambda^{3}_{2}+\bar \lambda^{3}_{3})(\bar \lambda^{4}_{1} +\bar \lambda^{4}_{2}+\bar \lambda^{4}_{3})=0,$$
that is
\begin{equation}\label{3.1-50}
S\bar f_{5}-\bar f_{3}f_{4}=0.
\end{equation}

\noindent From \eqref{2.1-16} and \eqref{2.1-17} in Lemma \ref{lemma 2.2}, we have
\begin{equation*}
\begin{aligned}
&\sum_{i,j,k}h_{ijk}^{2}+S^{2}-\lambda f_{3}=0, \\
&2\sum_{i,j,k,l,m}h_{ijm}h_{jkm}h_{kl}h_{li}+\sum_{i,j,k,l,m}h_{ijm}h_{jk}h_{klm}h_{li}+Sf_{4}-\lambda f_{5}=0.
\end{aligned}
\end{equation*}
Thus,
\begin{equation*}
\begin{aligned}
&\sum_{i,j,k}\bar h_{ijk}^{2}+S^{2}-\lambda \bar f_{3}=0, \\
&2\sum_{i,j,k,l,m}\bar h_{ijm}\bar h_{jkm}\bar h_{kl}\bar h_{li}+\sum_{i,j,k,l,m}\bar h_{ijm}\bar h_{jk}\bar h_{klm}\bar h_{li}+Sf_{4}-\lambda \bar f_{5}=0.
\end{aligned}
\end{equation*}
Especially,
\begin{equation}\label{3.1-51}
\begin{aligned}
&\bar h^{2}_{123}=-\frac{1}{6}(S^{2}-\lambda \bar f_{3}), \\
&\bar h^{2}_{123}=\frac{-(Sf_{4}-\lambda \bar f_{5})}{\bar H^{2}+3S}.
\end{aligned}
\end{equation}

\noindent From \eqref{3.1-51}, we obtain
\begin{equation*}
\lambda\bigg(6\bar f_{5}-\bar f_{3}(\bar H^{2}+3S)\bigg)=6Sf_{4}-S^{2}(\bar H^{2}+3S).
\end{equation*}

\noindent Supposing $$6\bar f_{5}-\bar f_{3}(\bar H^{2}+3S)=0,$$
we obtain
\begin{equation}\label{3.1-52}
6f_{4}=S(\bar H^{2}+3S).
\end{equation}

\noindent From Lemma \ref{lemma 2.4},
we have
\begin{equation}\label{3.1-53}
\begin{aligned}
&f_{4}=\frac{4}{3}\bar H\bar f_{3}-\bar H^{2}S+\frac{1}{6}\bar H^{4}+\frac{1}{2}S^{2}, \\
&\bar f_{5}=\frac{5}{6}\bar H^{2}\bar f_{3}+\frac{5}{6}S\bar f_{3}-\frac{5}{6}\bar H^{3}S+\frac{1}{6}\bar H^{5}.
\end{aligned}
\end{equation}

\noindent From \eqref{3.1-50} and \eqref{3.1-53}, we obtain
\begin{equation}\label{3.1-54}
8\bar H\bar f_{3}^{2}+(\bar H^{4}-11\bar H^{2}S-2S^{2})\bar f_{3}+5\bar H^{3}S^{2}-\bar H^{5}S=0.
\end{equation}

\noindent From \eqref{3.1-52} and \eqref{3.1-53}, we obtain
\begin{equation*}
8\bar H\bar f_{3}-7\bar H^{2}S+\bar H^{4}=0,
\end{equation*}
that is,
\begin{equation}\label{3.1-55}
\bar f_{3}=\frac{7}{8}\bar HS-\frac{1}{8}\bar H^{3}.
\end{equation}

\noindent From \eqref{3.1-54} and \eqref{3.1-55}, we obtain
\begin{equation*}
\bar HS(2\bar H^{4}-7\bar H^{2}S+7S^{2})=\bar HS\bigg(2(\bar H^{2}-\frac{7}{4}S)^{2}+\frac{7}{8}S^{2}\bigg)=0,
\end{equation*}
which is impossible. Then we have
\begin{equation}\label{3.1-56}
6\bar f_{5}-\bar f_{3}(\bar H^{2}+3S)\neq 0, \ \ \lambda=\frac{6Sf_{4}-S^{2}(\bar H^{2}+3S)}{6\bar f_{5}-\bar f_{3}(\bar H^{2}+3S)}.
\end{equation}

\noindent From \eqref{3.1-51} and \eqref{3.1-56}, we obtain
\begin{equation*}
\begin{aligned}
\bar h_{123}^{2}
&=-\frac{1}{6}(S^{2}-\lambda \bar f_{3})\\
&=-\frac{1}{6} \bigg(S^{2}-\bar f_{3}\cdot\frac{6Sf_{4}-S^{2}(\bar H^{2}+3S)}{6\bar f_{6}-\bar f_{3}(\bar H^{2}+3S)}\bigg)\\
&=-S\bigg(\frac{S\bar f_{5}-f_{4}\bar f_{3}}{6\bar f_{5}-\bar f_{3}(\bar H^{2}+3S)}\bigg)\\
&=0,
\end{aligned}
\end{equation*}
where $S\bar f_{5}-\bar f_{3}f_{4}=0$.

\noindent That is, $$\bar h_{123}=0, \ \ \bar h_{ijk}=0, \ \ \text{for } \ i,j,k=1, 2, 3.$$

\noindent Supposing $\bar H-\lambda=0$, from $\bar h_{123}=0$ and \eqref{3.1-48},
we obtain
\begin{equation*}
\bar h_{ijkl}=0, \ \ \text{for } \ i,j,k=1, 2, 3.
\end{equation*}

\noindent From \eqref{2.1-18} and \eqref{2.1-19} in lemma \ref{lemma 2.3}, we have
\begin{equation*}
\begin{aligned}
0=&\lim_{t\rightarrow\infty}\frac{1}{2}\Delta_{-V}\sum_{i, j,k}(h_{ijk})^{2}(p_{t}) \\
 =&\frac{9}{2}\lambda S\bar h_{11}\bar h_{22}\bar h_{33}-\frac{3}{2}\lambda^{2}\sum_{k}(\bar h_{22}\bar h_{33}\bar h^{2}_{1k}+\bar h_{11}\bar h_{33}\bar h^{2}_{2k}+\bar h_{11}\bar h_{22}\bar h^{2}_{3k}) \\
 =&\frac{3}{2}\lambda \bar\lambda_{1}\bar \lambda_{2}\bar \lambda_{3}(3S-\lambda \bar H) \\
 =&\frac{3}{2}\lambda \bar\lambda_{1}\bar \lambda_{2}\bar \lambda_{3}(3S-\bar H^{2}),
\end{aligned}
\end{equation*}
where $\bar H-\lambda=0$ and $\bar h_{ijk}=0, \ \bar h_{ijkl}=0, \ i,j,k,l=1, 2, 3$.

\noindent Therefore,
\begin{equation*}
3S-\bar H^{2}=0, \ \ \bar \lambda_{1}=\bar \lambda_{2}=\bar \lambda_{3}.
\end{equation*}
This contradicts the hypothesis. We have $$\bar H-\lambda\neq 0.$$

\noindent From  $\bar h_{123}=0$ and \eqref{3.1-51}, we have
\begin{equation}\label{3.1-57}
\lambda=\frac{S^{2}}{\bar f_{3}}, \ \ \frac{\bar H-\lambda}{\bar H}=\frac{\bar H \bar f_{3}-S^{2}}{\bar H \bar f_{3}}.
\end{equation}

\noindent From $S=constant$ and \eqref{2.1-16} in Lemma \ref{lemma 2.2}, we have
\begin{equation*}
\begin{aligned}
&2\sum_{i,j,k}h_{ijk}h_{ijkl}-\lambda\nabla_{l}f_{3}=0, \\
&2\sum_{i,j,k}h_{ijk}h_{ijklm}+2\sum_{i,j,k}h_{ijkm}h_{ijkl}-\lambda\nabla_{m}\nabla_{l}f_{3}=0,
\ \ \text{for } \ l,m=1, 2, 3.
\end{aligned}
\end{equation*}
Thus,
\begin{equation*}
\sum_{i,j,k}\bar h_{ijk}\bar h_{ijklm}+\sum_{i,j,k}\bar h_{ijkm}\bar h_{ijkl}-\frac{1}{2}\lambda\lim_{t\rightarrow\infty}\nabla_{m}\nabla_{l}f_{3}(p_{t})=0,
\ \ \text{for } \ l,m=1, 2, 3.
\end{equation*}
Especially,
\begin{equation}\label{3.1-58}
\begin{aligned}
&\bar h^{2}_{1111}+3\bar h^{2}_{2211}+3\bar h^{2}_{3311}-\frac{1}{2}\lambda\lim_{t\rightarrow\infty}\nabla_{1}\nabla_{1}f_{3}(p_{t})=0,\\
&\bar h^{2}_{2222}+3\bar h^{2}_{1122}+3\bar h^{2}_{3322}-\frac{1}{2}\lambda\lim_{t\rightarrow\infty}\nabla_{2}\nabla_{2}f_{3}(p_{t})=0,\\
&\bar h^{2}_{3333}+3\bar h^{2}_{1133}+3\bar h^{2}_{2233}-\frac{1}{2}\lambda\lim_{t\rightarrow\infty}\nabla_{3}\nabla_{3}f_{3}(p_{t})=0.
\end{aligned}
\end{equation}

\noindent From $f_{4}=constant$ and \eqref{2.1-21} in Lemma \ref{lemma 2.4}, we have

\begin{equation*}
(\frac{4}{3}\bar f_{3}-2S\bar H+\frac{2}{3}\bar H^{3})\bar H_{,kl}+\frac{4}{3}\bar H\lim_{t\rightarrow\infty}\nabla_{l}\nabla_{k}f_{3}(p_{t})=0,
\end{equation*}
and then,
\begin{equation*}
\lim_{t\rightarrow\infty}\nabla_{l}\nabla_{k}f_{3}(p_{t})=-\frac{\bar f_{3}-\frac{3}{2}S\bar H+\frac{1}{2}\bar H^{3}}{\bar H}\cdot\bar H_{,kl}=-\frac{3\bar\lambda_{1}\bar\lambda_{2}\bar\lambda_{3}}{\bar H}\bar H_{,kl}.
\end{equation*}
Therefore,
\begin{equation}\label{3.1-59}
\begin{aligned}
&-\frac{1}{2}\lambda\lim_{t\rightarrow\infty}\nabla_{1}\nabla_{1}f_{3}(p_{t})=\lambda\cdot\frac{3\bar\lambda_{1}\bar\lambda_{2}\bar\lambda_{3}}{2\bar H}\bar H_{,11}=\frac{3\lambda(\bar H-\lambda)}{2\bar H}\cdot\bar \lambda^{3}_{1}\bar\lambda_{2}\bar\lambda_{3},\\
&-\frac{1}{2}\lambda\lim_{t\rightarrow\infty}\nabla_{2}\nabla_{2}f_{3}(p_{t})=\lambda\cdot\frac{3\bar\lambda_{1}\bar\lambda_{2}\bar\lambda_{3}}{2\bar H}\bar H_{,22}=\frac{3\lambda(\bar H-\lambda)}{2\bar H}\cdot\bar \lambda_{1}\bar\lambda^{3}_{2}\bar\lambda_{3},\\
&-\frac{1}{2}\lambda\lim_{t\rightarrow\infty}\nabla_{3}\nabla_{3}f_{3}(p_{t})=\lambda\cdot\frac{3\bar\lambda_{1}\bar\lambda_{2}\bar\lambda_{3}}{2\bar H}\bar H_{,33}=\frac{3\lambda(\bar H-\lambda)}{2\bar H}\cdot\bar \lambda_{1}\bar\lambda_{2}\bar\lambda^{3}_{3}.
\end{aligned}
\end{equation}

\noindent From $\bar h_{123}=0$ and \eqref{3.1-48}, we have
\begin{equation}\label{3.1-60}
\begin{cases}
\begin{aligned}
&\bar h_{1111}=\frac{\bar \lambda_{2}\bar \lambda_{3}(\bar \lambda_{2}+\bar \lambda_{3})}{(\bar \lambda_{1}-\bar \lambda_{2})(\bar \lambda_{1}-\bar \lambda_{3})}\cdot\frac{(\bar H-\lambda)\bar \lambda^{2}_{1}}{\bar H}, \\

&\bar h_{2211}=\frac{\bar \lambda_{1}\bar \lambda_{3}(\bar \lambda_{1}+\bar \lambda_{3})}{(\bar \lambda_{2}-\bar \lambda_{1})(\bar \lambda_{2}-\bar \lambda_{3})}\cdot\frac{(\bar H-\lambda)\bar \lambda^{2}_{1}}{\bar H}, \\

&\bar h_{3311}=\frac{\bar \lambda_{1}\bar \lambda_{2}(\bar \lambda_{1}+\bar \lambda_{2})}{(\bar \lambda_{3}-\bar \lambda_{1})(\bar \lambda_{3}-\bar \lambda_{2})}\cdot\frac{(\bar H-\lambda)\bar \lambda^{2}_{1}}{\bar H}, \\

&\bar h_{2222}=\frac{\bar \lambda_{1}\bar \lambda_{3}(\bar \lambda_{1}+\bar \lambda_{3})}{(\bar \lambda_{2}-\bar \lambda_{1})(\bar \lambda_{2}-\bar \lambda_{3})}\cdot\frac{(\bar H-\lambda)\bar \lambda^{2}_{2}}{\bar H}, \\

&\bar h_{1122}=\frac{\bar \lambda_{2}\bar \lambda_{3}(\bar \lambda_{2}+\bar \lambda_{3})}{(\bar \lambda_{1}-\bar \lambda_{2})(\bar \lambda_{1}-\bar \lambda_{3})}\cdot\frac{(\bar H-\lambda)\bar \lambda^{2}_{2}}{\bar H}, \\

&\bar h_{3322}=\frac{\bar \lambda_{1}\bar \lambda_{2}(\bar \lambda_{1}+\bar \lambda_{2})}{(\bar \lambda_{3}-\bar \lambda_{1})(\bar \lambda_{3}-\bar \lambda_{2})}\cdot\frac{(\bar H-\lambda)\bar \lambda^{2}_{2}}{\bar H}, \\

&\bar h_{3333}=\frac{\bar \lambda_{1}\bar \lambda_{2}(\bar \lambda_{1}+\bar \lambda_{2})}{(\bar \lambda_{3}-\bar \lambda_{1})(\bar \lambda_{3}-\bar \lambda_{2})}\cdot\frac{(\bar H-\lambda)\bar \lambda^{2}_{3}}{\bar H}, \\

&\bar h_{1133}=\frac{\bar \lambda_{2}\bar \lambda_{3}(\bar \lambda_{2}+\bar \lambda_{3})}{(\bar \lambda_{1}-\bar \lambda_{2})(\bar \lambda_{1}-\bar \lambda_{3})}\cdot\frac{(\bar H-\lambda)\bar \lambda^{2}_{3}}{\bar H}, \\

&\bar h_{2233}=\frac{\bar \lambda_{1}\bar \lambda_{3}(\bar \lambda_{1}+\bar \lambda_{3})}{(\bar \lambda_{2}-\bar \lambda_{1})(\bar \lambda_{2}-\bar \lambda_{3})}\cdot\frac{(\bar H-\lambda)\bar \lambda^{2}_{3}}{\bar H}.
\end{aligned}
\end{cases}
\end{equation}

\noindent From \eqref{3.1-57}, \eqref{3.1-58}, \eqref{3.1-59} and \eqref{3.1-60}, we have
\begin{equation}\label{3.1-61}
\begin{cases}
\begin{aligned}
&\frac{\bar \lambda_{1}(\bar H \bar f_{3}-S^{2})}{\bar H \bar f_{3}}\cdot\bigg(
\frac{\bar \lambda^{2}_{2}\bar \lambda^{2}_{3}(\bar \lambda_{2}+\bar \lambda_{3})^{2}}{(\bar \lambda_{1}-\bar \lambda_{2})^{2}(\bar \lambda_{1}-\bar \lambda_{3})^{2}}+\frac{3\bar \lambda^{2}_{1}\bar \lambda^{2}_{3}(\bar \lambda_{1}+\bar \lambda_{3})^{2}}{(\bar \lambda_{2}-\bar \lambda_{1})^{2}(\bar \lambda_{2}-\bar \lambda_{3})^{2}}\\
&+\frac{3\bar \lambda^{2}_{1}\bar \lambda^{2}_{2}(\bar \lambda_{1}+\bar \lambda_{2})^{2}}{(\bar \lambda_{3}-\bar \lambda_{1})^{2}(\bar \lambda_{3}-\bar \lambda_{2})^{2}}
\bigg)+\frac{3\bar\lambda_{2}\bar\lambda_{3}S^{2}}{2\bar f_{3}}=0, \\

&\frac{\bar \lambda_{2}(\bar H \bar f_{3}-S^{2})}{\bar H \bar f_{3}}\cdot\bigg(
\frac{\bar \lambda^{2}_{1}\bar \lambda^{2}_{3}(\bar \lambda_{1}+\bar \lambda_{3})^{2}}{(\bar \lambda_{2}-\bar \lambda_{1})^{2}(\bar \lambda_{2}-\bar \lambda_{3})^{2}}+\frac{3\bar \lambda^{2}_{2}\bar \lambda^{2}_{3}(\bar \lambda_{2}+\bar \lambda_{3})^{2}}{(\bar \lambda_{1}-\bar \lambda_{2})^{2}(\bar \lambda_{1}-\bar \lambda_{3})^{2}}\\
&+\frac{3\bar \lambda^{2}_{1}\bar \lambda^{2}_{2}(\bar \lambda_{1}+\bar \lambda_{2})^{2}}{(\bar \lambda_{3}-\bar \lambda_{1})^{2}(\bar \lambda_{3}-\bar \lambda_{2})^{2}}
\bigg)+\frac{3\bar\lambda_{1}\bar\lambda_{3}S^{2}}{2\bar f_{3}}=0, \\

&\frac{\bar \lambda_{3}(\bar H \bar f_{3}-S^{2})}{\bar H \bar f_{3}}\cdot\bigg(
\frac{\bar \lambda^{2}_{1}\bar \lambda^{2}_{2}(\bar \lambda_{1}+\bar \lambda_{2})^{2}}{(\bar \lambda_{3}-\bar \lambda_{1})^{2}(\bar \lambda_{3}-\bar \lambda_{2})^{2}}+\frac{3\bar \lambda^{2}_{2}\bar \lambda^{2}_{3}(\bar \lambda_{2}+\bar \lambda_{3})^{2}}{(\bar \lambda_{1}-\bar \lambda_{2})^{2}(\bar \lambda_{1}-\bar \lambda_{2})^{2}}\\
&+\frac{3\bar \lambda^{2}_{1}\bar \lambda^{2}_{3}(\bar \lambda_{1}+\bar \lambda_{3})^{2}}{(\bar \lambda_{2}-\bar \lambda_{1})^{2}(\bar \lambda_{2}-\bar \lambda_{3})^{2}}
\bigg)+\frac{3\bar\lambda_{1}\bar\lambda_{2}S^{2}}{2\bar f_{3}}=0.
\end{aligned}
\end{cases}
\end{equation}
And then,
\begin{equation*}
\begin{cases}
\begin{aligned}
&2\bar \lambda_{1}(\bar H \bar f_{3}-S^{2})\cdot\bigg(
\bar \lambda^{2}_{2}\bar \lambda^{2}_{3}(\bar \lambda^{2}_{2}-\bar \lambda^{2}_{3})^{2}+3\bar \lambda^{2}_{1}\bar \lambda^{2}_{3}(\bar \lambda^{2}_{1}-\bar \lambda^{2}_{3})^{2}+3\bar \lambda^{2}_{1}\bar \lambda^{2}_{2}(\bar \lambda^{2}_{1}-\bar \lambda^{2}_{2})^{2}\bigg)\\
&+3\bar\lambda_{2}\bar\lambda_{3}\bar H S^{2}(\bar \lambda_{1}-\bar \lambda_{2})^{2}(\bar \lambda_{1}-\bar \lambda_{3})^{2}(\bar \lambda_{2}-\bar \lambda_{3})^{2}=0, \ \ (1)\\

&2\bar \lambda_{2}(\bar H \bar f_{3}-S^{2})\cdot\bigg(
\bar \lambda^{2}_{1}\bar \lambda^{2}_{3}(\bar \lambda^{2}_{1}-\bar \lambda^{2}_{3})^{2}+3\bar \lambda^{2}_{2}\bar \lambda^{2}_{3}(\bar \lambda^{2}_{2}-\bar \lambda^{2}_{3})^{2}+3\bar \lambda^{2}_{1}\bar \lambda^{2}_{2}(\bar \lambda^{2}_{1}-\bar \lambda^{2}_{2})^{2}\bigg)\\
&+3\bar\lambda_{1}\bar\lambda_{3}\bar H S^{2}(\bar \lambda_{1}-\bar \lambda_{2})^{2}(\bar \lambda_{1}-\bar \lambda_{3})^{2}(\bar \lambda_{2}-\bar \lambda_{3})^{2}=0, \ \ (2)\\

&2\bar \lambda_{3}(\bar H \bar f_{3}-S^{2})\cdot\bigg(
\bar \lambda^{2}_{1}\bar \lambda^{2}_{2}(\bar \lambda^{2}_{1}-\bar \lambda^{2}_{2})^{2}+3\bar \lambda^{2}_{2}\bar \lambda^{2}_{3}(\bar \lambda^{2}_{2}-\bar \lambda^{2}_{3})^{2}+3\bar \lambda^{2}_{1}\bar \lambda^{2}_{3}(\bar \lambda^{2}_{1}-\bar \lambda^{2}_{3})^{2}\bigg)\\
&+3\bar\lambda_{1}\bar\lambda_{2}\bar H S^{2}(\bar \lambda_{1}-\bar \lambda_{2})^{2}(\bar \lambda_{1}-\bar \lambda_{3})^{2}(\bar \lambda_{2}-\bar \lambda_{3})^{2}=0. \ \ (3)
\end{aligned}
\end{cases}
\end{equation*}

\noindent By computing  $\bar \lambda_{1}\times (1)-\bar \lambda_{2}\times (2)$, we have
\begin{equation*}
\begin{aligned}
&2(\bar \lambda^{2}_{1}-\bar \lambda^{2}_{2})(\bar H \bar f_{3}-S^{2})\bigg(3\bar \lambda^{2}_{1}\bar \lambda^{2}_{2}(\bar \lambda^{2}_{1}-\bar \lambda^{2}_{2})^{2}+3\bar \lambda^{2}_{1}\bar \lambda^{2}_{3}(\bar \lambda^{2}_{1}-\bar \lambda^{2}_{3})^{2}+3\bar \lambda^{2}_{2}\bar \lambda^{2}_{3}(\bar \lambda^{2}_{2}-\bar \lambda^{2}_{3})^{2}\\
&+2\bar \lambda^{2}_{1}\bar \lambda^{2}_{2}\bar \lambda^{2}_{3}(\bar \lambda^{2}_{1}+\bar \lambda^{2}_{2}-2\bar \lambda^{2}_{3})\bigg)=0.
\end{aligned}
\end{equation*}

\noindent Supposing  $\bar H \bar f_{3}-S^{2}=0$, from \eqref{3.1-54}, we obtain
\begin{equation}\label{3.1-62}
\begin{aligned}
0=&8\bar H\bar f_{3}^{2}+(\bar H^{4}-11\bar H^{2}S-2S^{2})\bar f_{3}+5\bar H^{3}S^{2}-\bar H^{5}S\\
=&\frac{S}{\bar H}(6S^{3}-11\bar H^{2}S^{2}+6\bar H^{4}S-\bar H^{6})\\
=&\frac{S}{\bar H}(S-\bar H^{2})(2S-\bar H^{2})(3S-\bar H^{2}).
\end{aligned}
\end{equation}

\noindent From $\lambda_{1}\neq\lambda_{2}\neq\lambda_{3}\neq\lambda_{1}$, we obtain
$$
H^{2}=(\lambda_{1}+\lambda_{2}+\lambda_{3})^{2}< 3(\lambda_{1}^{2} +\lambda_{2}^{2} +\lambda_{3}^{2})=3S.
$$
Hence,
$$
H^{2}<3S.
$$

\noindent From \eqref{3.1-62}, we obtain that $S-\bar H^{2}=0$ or $2S-\bar H^{2}=0$.
Besides, for $n=3$, we have $\bar f_{3}=\frac{\bar H}{2}(3S-\bar H^{2})+3\bar \lambda_{1}\bar \lambda_{2}\bar \lambda_{3}$.

\noindent When $S-\bar H^{2}=0$, we have that
\begin{equation*}
\begin{aligned}
&\bar f_{3}=\frac{S^{2}}{\bar H}=\bar H^{3},\\
&\bar f_{3}=\frac{\bar H}{2}(3S-\bar H^{2})+3\bar \lambda_{1}\bar \lambda_{2}\bar \lambda_{3}=\bar H^{3}+3\bar \lambda_{1}\bar \lambda_{2}\bar \lambda_{3}.
\end{aligned}
\end{equation*}
And then, $\bar \lambda_{1}\bar \lambda_{2}\bar \lambda_{3}=0$. This contradicts the hypothesis.

\noindent When $2S-\bar H^{2}=0$, we have that
\begin{equation*}
\begin{aligned}
&\bar f_{3}=\frac{S^{2}}{\bar H}=\frac{\bar H^{3}}{4},\\
&\bar f_{3}=\frac{\bar H}{2}(3S-\bar H^{2})+3\bar \lambda_{1}\bar \lambda_{2}\bar \lambda_{3}=\frac{\bar H^{3}}{4}+3\bar \lambda_{1}\bar \lambda_{2}\bar \lambda_{3}.
\end{aligned}
\end{equation*}
And then, $\bar \lambda_{1}\bar \lambda_{2}\bar \lambda_{3}=0$. This contradicts the hypothesis.
Hence, $$\bar H \bar f_{3}-S^{2}\neq 0.$$

\noindent Supposing  $\bar \lambda^{2}_{1}-\bar \lambda^{2}_{2}=0$, that is $\bar \lambda_{1}=-\bar \lambda_{2}$.

\noindent From \eqref{3.1-50}, we obtain
\begin{equation*}
\begin{aligned}
0=&S\bar f_{5}-\bar f_{3}f_{4}\\
=&(2\bar \lambda^{2}_{1}+\bar \lambda^{2}_{3})\bar \lambda^{5}_{3}-(2\bar \lambda^{4}_{1}+\bar \lambda^{4}_{3})\bar \lambda^{3}_{3}\\
=&2\bar \lambda^{2}_{1}\bar \lambda^{3}_{3}(\bar \lambda^{2}_{3}-\bar \lambda^{2}_{1}),
\end{aligned}
\end{equation*}
which implies $\bar \lambda^{2}_{1}=\bar \lambda^{2}_{3}$. Then $\bar \lambda_{1}=\bar \lambda_{3}$ or $\bar \lambda_{1}=-\bar \lambda_{3}=-\bar \lambda_{2}$, which is a contradiction.
Hence,
\begin{equation}\label{3.1-63}
\begin{aligned}
&3\bar \lambda^{2}_{1}\bar \lambda^{2}_{2}(\bar \lambda^{2}_{1}-\bar \lambda^{2}_{2})^{2}+3\bar \lambda^{2}_{1}\bar \lambda^{2}_{3}(\bar \lambda^{2}_{1}-\bar \lambda^{2}_{3})^{2}+3\bar \lambda^{2}_{2}\bar \lambda^{2}_{3}(\bar \lambda^{2}_{2}-\bar \lambda^{2}_{3})^{2}\\
&+2\bar \lambda^{2}_{1}\bar \lambda^{2}_{2}\bar \lambda^{2}_{3}(\bar \lambda^{2}_{1}+\bar \lambda^{2}_{2}-2\bar \lambda^{2}_{3})=0.
\end{aligned}
\end{equation}
\noindent Similarity, by computing  $\bar \lambda_{2}\times (2)-\bar \lambda_{3}\times (3)$, we have
\begin{equation*}
\begin{aligned}
&2(\bar \lambda^{2}_{2}-\bar \lambda^{2}_{3})(\bar H \bar f_{3}-S^{2})\bigg(3\bar \lambda^{2}_{1}\bar \lambda^{2}_{2}(\bar \lambda^{2}_{1}-\bar \lambda^{2}_{2})^{2}+3\bar \lambda^{2}_{1}\bar \lambda^{2}_{3}(\bar \lambda^{2}_{1}-\bar \lambda^{2}_{3})^{2}+3\bar \lambda^{2}_{2}\bar \lambda^{2}_{3}(\bar \lambda^{2}_{2}-\bar \lambda^{2}_{3})^{2}\\
&+2\bar \lambda^{2}_{1}\bar \lambda^{2}_{2}\bar \lambda^{2}_{3}(\bar \lambda^{2}_{2}+\bar \lambda^{2}_{3}-2\bar \lambda^{2}_{1})\bigg)=0,
\end{aligned}
\end{equation*}
which implies \begin{equation}\label{3.1-64}
\begin{aligned}
&3\bar \lambda^{2}_{1}\bar \lambda^{2}_{2}(\bar \lambda^{2}_{1}-\bar \lambda^{2}_{2})^{2}+3\bar \lambda^{2}_{1}\bar \lambda^{2}_{3}(\bar \lambda^{2}_{1}-\bar \lambda^{2}_{3})^{2}+3\bar \lambda^{2}_{2}\bar \lambda^{2}_{3}(\bar \lambda^{2}_{2}-\bar \lambda^{2}_{3})^{2}\\
&+2\bar \lambda^{2}_{1}\bar \lambda^{2}_{2}\bar \lambda^{2}_{3}(\bar \lambda^{2}_{2}+\bar \lambda^{2}_{3}-2\bar \lambda^{2}_{1})=0.
\end{aligned}
\end{equation}

\noindent From \eqref{3.1-63} and \eqref{3.1-64}, we have
$$\bar \lambda^{2}_{1}+\bar \lambda^{2}_{2}-2\bar \lambda^{2}_{3}=\bar \lambda^{2}_{2}+\bar \lambda^{2}_{3}-2\bar \lambda^{2}_{1}.$$
That is, $\bar \lambda_{1}=-\bar \lambda_{3}$.

\noindent From \eqref{3.1-50} and $\bar \lambda_{1}=-\bar \lambda_{3}$, we obtain
\begin{equation*}
\begin{aligned}
0=&S\bar f_{5}-\bar f_{3}f_{4}\\
=&(2\bar \lambda^{2}_{1}+\bar \lambda^{2}_{2})\bar \lambda^{5}_{2}-(2\bar \lambda^{4}_{1}+\bar \lambda^{4}_{2})\bar \lambda^{3}_{2}\\
=&2\bar \lambda^{2}_{1}\bar \lambda^{3}_{2}(\bar \lambda^{2}_{2}-\bar \lambda^{2}_{1}),
\end{aligned}
\end{equation*}
Then $\bar \lambda_{1}=-\bar \lambda_{3}=-\bar \lambda_{2}$ and $\bar \lambda_{2}=\bar \lambda_{3}$, which is a contradiction.
\end{proof}

\begin{theorem}\label{theorem 3.3}
For a $3$-dimensional complete $\lambda$-translator $x:M^{3}\rightarrow \mathbb{R}^{4}_{1}$ with non-zero constant squared norm $S$ of the second fundamental form and constant $f_{4}$, where $S=\sum_{i,j}h_{ij}^2$ and $f_{4}=\sum_{i,j,k,l}h_{ij}h_{jk}h_{kl}h_{li}$,
we have either
\begin{enumerate}
\item $\lambda^{2}=S$ and $\inf H^{2}=S$, or
\item $\lambda^{2}=2S$ and $\inf H^{2}=2S$, or
\item $\lambda^{2}=3S$ and $\inf H^{2}=3S$.
\end{enumerate}
\end{theorem}

\begin{proof}
We apply the generalized maximum principle for the operator $\Delta_{-V}$
to the function $-H^{2}$. Thus, there exists a sequence $\{p_{t}\}$ in $M^{3}$ such that
\begin{equation*}
\lim_{t\rightarrow\infty} H^{2}(p_{t})=\inf H^{2}=\bar H^{2}, \ \
\lim_{t\rightarrow\infty} |\nabla H^{2}(p_{t})|=0, \ \
\lim_{t\rightarrow\infty}\Delta_{-V} H^{2}(p_{t})\geq 0,
\end{equation*}
that is,
\begin{equation*}
\begin{cases}
\begin{aligned}
&\lim_{t\rightarrow\infty} H^{2}(p_{t})=\sup H^{2}=\bar H^{2},\quad
\lim_{t\rightarrow\infty} |\nabla H^{2}(p_{t})|=0,\\
&0\leq
\lim_{t\rightarrow\infty} |\nabla H|^{2}(p_{t})+S(\bar H-\lambda)\bar H.
\end{aligned}
\end{cases}
\end{equation*}
\vskip2mm
\noindent
By taking the limit and making use of the same assertion as in Theorem \ref{theorem 3.1}, we can prove
$\inf H^2>0$.
Hence, without loss of the generality, we can assume
$$
\lim_{t\rightarrow\infty}h_{ijl}(p_{t})=\bar h_{ijl}, \quad \lim_{t\rightarrow\infty}h_{ij}(p_{t})=\bar h_{ij}=\bar \lambda \delta_{ij},
\quad \lim_{t\rightarrow\infty}h_{ijkl}(p_{t})=\bar h_{ijkl},
$$
for $i, j, k, l=1, 2, 3$.
\noindent

\noindent By making use of the same assertion as in the proof of the
Theorem \ref{theorem 3.2}, we will discuss three cases.

\noindent Case 1: The values of the principal curvature $\bar \lambda_1$, $\bar \lambda_{2}$ and $\bar \lambda_3$ are not equal to each other.

\noindent This case does not exist.

\noindent Case 2: Two of the values of the principal curvature $\bar \lambda_1$, $\bar \lambda_2$ and $\bar \lambda_3$  are equal.

\noindent There are two scenarios:

\noindent First, one is not zero and the other two are equal to zero. We have
\begin{equation*}
\lambda^{2}=\bar H^{2}=\inf H^{2}, \ \ \lambda^{2}=S;
\end{equation*}
Second, one is zero and the other two are equal and not zero. We have
\begin{equation*}
  \lambda^{2}=\bar H^{2}=\inf H^{2}, \ \ \lambda^{2}=2S;
\end{equation*}

\noindent Case 3: The values of the principal curvature $\bar \lambda_1$, $\bar \lambda_{2}$ and $\bar \lambda_3$ are all equal.

\noindent We have
\begin{equation*}
\inf  H^{2}=(\bar \lambda_{1}+\bar \lambda_{2}+\bar \lambda_{3})^2=3S.
\end{equation*}
Since
$$0\leq 3S-H^{2}\leq \sup (3S-H^{2})=3S-\inf H^{2}=0.$$
Namely,  we obtain $H$ is constant.
Hence, we conclude from \eqref{2.1-15}
$$\lambda=H, \ \ \lambda^{2}=3S.$$
The proof of Theorem \ref{theorem 3.3} is finished.
\end{proof}

\vskip3mm
\noindent
{\it Proof of Theorem \ref{theorem 1.1}}. If $S=0$, we know that $x: M^{3}\to \mathbb{R}^{4}_{1}$ is a space-like affine plane $\mathbb{R}^{3}_{1}$, not necessarily passing through the origin. If $S \neq 0$,
from Theorem \ref{theorem 3.2} and Theorem \ref{theorem 3.3}, we have
\begin{enumerate}
\item $\lambda^{2}=S$ and $\sup H^{2}=\inf H^{2}=S$, or
\item $\lambda^{2}=2S$ and $\sup H^{2}=\inf H^{2}=2S$, or
\item $\lambda^{2}=3S$ and $\sup H^{2}=\inf H^{2}=3S$.
\end{enumerate}
It follows that the mean curvature $H$ and the principal curvature must be a constant. From \eqref{1.1-2} and \eqref{2.1-14}, we have
$$\lambda=H, \ \ \langle T, \vec N\rangle=0.$$ So the nonzero constant
vector $T = T^{T}$ is tangent to $x(M^{3})$ at each point of $M^{3}$. It follows that $x(M^{3})$ consists of a family of parallel planes
in $\mathbb{R}^{4}_{1}$ and thus, up to an isometry of $\mathbb{R}^{4}_{1}$, it is a cylinder $\mathbb{H}^{1}(a_{1}) \times \mathbb{R}^{2}$ or $\mathbb{H}^{2}(a_{2}) \times \mathbb{R}^{1}$ for $a_{1}>0$ and $a_{2}>0$, where $\mathbb{H}^{1}(a_{1})$ and $\mathbb{H}^{2}(a_{2})$ are hyperbolic curve and hyperboloid respectively. Besides, the parameters $a_{1}$ and $a_{2}$ can be determined by $\lambda$ via the defining equation \eqref{1.1-2}. By an easy computation, we have that $ \lambda>0$,  $a_{1}=\frac{1}{\lambda}$ and $a_{2}=\frac{2}{\lambda}$. Theorem \ref{theorem 1.1} is proved.

\end{document}